\documentclass[final,leqno]{siamltex1213}

\usepackage[numbers]{natbib}
\usepackage[dvips]{hyperref}
\usepackage{amsmath}
\usepackage{amsfonts}
\usepackage{amssymb}
\usepackage{commath}
\usepackage{graphicx}
\usepackage{subcaption}

\title{Three-field block-preconditioners for models of coupled
  magma/mantle dynamics\thanks{This work was supported by grants
    NE/I026995/1 and NE/I023929/1 from the UK Natural Environment
    Research Council}}

\author{Sander Rhebergen\thanks{Mathematical Institute, University of
    Oxford, Andrew Wiles Building, Radcliffe Observatory Quarter,
    Woodstock Road, Oxford OX2 6GG, United Kingdom and Department of
    Earth Sciences, University of Oxford, South Parks Road, Oxford OX1
    3AN, United Kingdom (\url{sander.rhebergen@maths.ox.ac.uk}).}
  \and Garth~N.~Wells\thanks{Department of Engineering, University of
    Cambridge, Trumpington Street, Cambridge CB2 1PZ, United Kingdom
    (\url{gnw20@cam.ac.uk}).}
  \and Andrew~J.~Wathen\thanks{Mathematical
    Institute, University of Oxford,
    Andrew Wiles Building, Radcliffe Observatory Quarter, Woodstock
    Road, Oxford OX2 6GG, United Kingdom
    (\url{andy.wathen@maths.ox.ac.uk}).}
  \and Richard~F.~Katz\thanks{Department of Earth Sciences,
    University of Oxford,
    South Parks Road, Oxford OX1 3AN, United Kingdom
    (\url{richard.katz@earth.ox.ac.uk}).}}

\begin{document}
\maketitle
\slugger{mms}{xxxx}{xx}{x}{x--x}
\begin{abstract}
  For a prescribed porosity, the coupled magma/mantle flow equations
  can be formulated as a two-field system of equations with velocity
  and pressure as unknowns. Previous work has shown that while optimal
  preconditioners for the two-field formulation can be obtained, the
  construction of preconditioners that are uniform with respect to
  model parameters is difficult. This limits the applicability of
  two-field preconditioners in certain regimes of practical interest.
  We address this issue by reformulating the governing equations as a
  three-field problem, which removes a term that was problematic in
  the two-field formulation in favour of an additional equation for a
  pressure-like field.  For the three-field problem, we develop and
  analyse new preconditioners and we show numerically that they are
  optimal in terms of problem size and less sensitive to model
  parameters, compared to the two-field preconditioner. This extends
  the applicability of optimal preconditioners for coupled
  mantle/magma dynamics into parameter regimes of physical interest.
\end{abstract}
\pagestyle{myheadings}
\thispagestyle{plain}
\markboth{{\footnotesize RHEBERGEN, WELLS,
    WATHEN, AND KATZ}}%
{{\footnotesize THREE-FIELD PRECONDITIONERS FOR COUPLED
    MAGMA/MANTLE DYNAMICS}}
\section{Introduction}

In this paper we consider numerical methods to efficiently solve the
linear system arising from the discretization of the equations for
coupled magma/mantle dynamics. These partial differential equations,
derived by \citet{McKenzie:1984}, model the two-phase flow of
partially molten regions of the Earth's mantle.  High ambient
temperatures enable slowly creeping flow of crystalline mantle rock,
and also permit melting of certain mantle minerals.  Melting produces
magma that resides within an interconnected network of pores amid the
mantle grains.  The governing equations describe the creeping flow of
the high-viscosity, solid mantle matrix and the porous flow of the
low-viscosity magma.  Although both the magma and mantle are
individually incompressible, the two-phase mixture permits
\emph{compaction}: non-zero convergence of the solid flux is
balanced by non-zero divergence of the magma flux (or vice
versa). Compaction therefore expels (or imbibes) magma locally,
changing the volume fraction of magma, termed the porosity. Compaction
flow is associated with a bulk viscosity and compaction stresses; it
gives rise to many of the interesting features of the coupled
dynamics. In a typical strategy for computing these dynamics,
solutions for the solid velocity field and the magma pressure field
are obtained for a fixed porosity field; the velocity and pressure are
then used to update the porosity. The magma velocity field can be
obtained diagnostically from the pressure and solid velocity fields.

Discretization of the elliptic equations for solid velocity and
pressure results in a linear system of algebraic equations that can be
expressed in a $2\times 2$ block-matrix format. Preconditioners are
crucial to efficiently solve the resulting system by iterative
methods; \citet{Rhebergen:2014} developed a diagonal block
preconditioner for this system and proved optimality with respect to
problem size. From numerical experiments, however, it was found that
performance of the preconditioner deteriorates at high values of the
bulk-to-shear-viscosity ratio. This parameter regime, which
corresponds to low values of porosity, is common in coupled
magma/mantle dynamics simulations: for example, at the boundary
between unmolten and partially molten mantle, where porosity varies
continuously from zero through values $~$1\%.  Such situations make
the preconditioner in \citep{Rhebergen:2014} of limited practical use.
At low values of porosity, the compaction stresses can become dominant
over the shear stresses.  In this case, the contribution of a
`grad-div' term in the momentum balance equation becomes significant;
such terms are known to be problematic for standard multigrid
methods. The two-field preconditioner in \citep{Rhebergen:2014} relied
on multigrid methods for the matrix blocks. The manifestation of the
problem was increasing Krylov solver iteration counts as the
bulk-to-shear-viscosity ratio increased.

In this paper, to circumvent the troublesome `grad-div' term, we
introduce a `compaction pressure' field, as was done by
\citet{Katz:2007} and \citet{Keller:2013}, and we reformulate the
problem as a three-field system.  This approach is also used in nearly
incompressible elasticity.  
Discretizing the model leads to a linear system of equations that may
be expressed in a $3\times 3$ block matrix format for which we develop
and analyse new block preconditioners in this work.  By introducing a
compaction pressure field, the size of the system is increased
compared to the $2\times 2$ block matrix. The relative increase in
degrees of freedom is limited, however, as the degrees of freedom for
the compaction pressure (like the degrees of freedom of the fluid
pressure) are fewer than the degrees of freedom of the solid velocity.
Moreover, we will demonstrate through numerical examples that
effective preconditioners for the three-field problem compensate for
the addition of an extra scalar field to the problem.

The remainder of this paper is structured as follows.  In
Section~\ref{s:governingequations} we present the two- and three-field
governing equations. We describe a weak formulation in
Section~\ref{s:weakform}, develop and analyse a lower block triangular
preconditioner in Section~\ref{s:preconditioning} and then discuss a
diagonal block preconditioner in Section~\ref{s:blockdiagonal_P}. In
Section~\ref{s:numericalsimulations} we verify our analysis by two and
three dimensional numerical simulations. Conclusions are drawn in
Section~\ref{s:conclusions}.

\section{Governing equations}
\label{s:governingequations}

On a domain $\Omega \subset \mathbb{R}^{d}$, where $1 \le d \le 3$,
for a given porosity field $\phi \in [0, 1]$ the non-dimensional
two-phase flow equations that describe coupled magma/mantle dynamics
are given by
\begin{subequations}
  \label{eq:mckenzie_nd}
  \begin{alignat}{1}
    \label{eq:mckenzie_nd_a}
    -\nabla\cdot\left(\eta{\bf D}{\bf u} \right)
    + \nabla p &=
    \nabla\left(\left(\zeta -
        \tfrac{1}{3}\eta\right)\nabla\cdot{\bf u} \right)
    +\phi{\bf e}_3,
    \\
    \label{eq:mckenzie_nd_b}
    \nabla\cdot{\bf u} &=
    \nabla\cdot\left(k\left(\nabla p
        - {\bf e}_3\right)\right),
  \end{alignat}
\end{subequations}
where $\eta > 0$ is the shear viscosity, ${\bf u}$ is the matrix
velocity, ${\bf D}{\bf u} = (\nabla{\bf u} + (\nabla{\bf u})^T)/2$ is
the total strain rate, $p$ is the dynamic pressure, $\zeta > 0$ is the
bulk viscosity, $k \ge 0$ is the permeability, and ${\bf e}_{3}$ is
the unit vector in the direction aligned with gravity (i.e., ${\bf
  e}_3 = (0, 1)$ when $d = 2$ and ${\bf e}_3=(0, 0, 1)$ when $d = 3$).
Throughout this paper we take the porosity $\phi$ to be a function of
${\bf x} \in \Omega$. Constitutive relations are required for the
permeability $k$, shear viscosity $\eta$ and bulk viscosity $\zeta$.
For now we just mention that $k$, $\eta$ and $\zeta$ are usually
functions of the porosity $\phi$. For more details on the derivation
of the two-phase flow equations~\eqref{eq:mckenzie_nd} we refer to
\citet{McKenzie:1984}. The non-dimensionalization of these equations
is presented in Appendix~\ref{ap:nondim-2-phase}.

In \citet{Rhebergen:2014} we studied \eqref{eq:mckenzie_nd} for the
restricted case of constant shear viscosity, constant bulk viscosity
and a spatially variable permeability that is independent of
porosity. These simplifications lead to the following system of
equations:
\begin{subequations}
  \label{eq:mckenzie_old}
  \begin{alignat}{1}
    \label{eq:mckenzie_old_a}
    -\nabla\cdot{\bf D}{\bf u}
    + \nabla \Tilde{p} &=
    \nabla \left(\alpha \nabla\cdot{\bf u} \right)
    +\phi{\bf e}_3/\eta,
    \\
    \label{eq:mckenzie_old_b}
    \nabla\cdot{\bf u}
    &=
    \nabla \cdot \left(\Tilde{k}\left(\nabla \Tilde{p}
        - {\bf e}_3/\eta\right)\right),
  \end{alignat}
\end{subequations}
where $\alpha = \zeta / \eta - 1/3$, $\Tilde{p} = p / \eta$, and
$\Tilde{k}=\eta k$. In \cite{Rhebergen:2014} we developed and analysed
a diagonal block preconditioner for a mixed finite element
discretization of \eqref{eq:mckenzie_old}. Combined with a Krylov
method, the preconditioner developed in \citep{Rhebergen:2014}
resulted in an optimal solver in terms of the problem size, but was
not uniform with respect to the model parameters. In particular, as
$\alpha$ increased the iteration count for the solver to reach a set
tolerance increased. This was attributed to the performance of standard
multigrid (geometric and algebraic) when the relative contribution of
the $\nabla(\nabla \cdot {\bf u})$ term becomes
significant~\citep{Rhebergen:2014}.

In this paper we develop new preconditioners for a reformulated system
of equations in which the $\nabla(\nabla \cdot {\bf u})$ term does not
appear explicitly. To achieve this we return to~\eqref{eq:mckenzie_nd}
and introduce the auxiliary variable~$p_c = -\zeta \nabla \cdot {\bf
  u}$, which allows us to write~\eqref{eq:mckenzie_nd} as
\begin{subequations}
  \label{eq:three-field}
  \begin{alignat}{1}
    \label{eq:three-field_a}
    -\nabla\cdot\left(\eta\left({\bf D}{\bf u}
        - \tfrac{1}{3}\nabla\cdot{\bf u}\mathbb{I}\right)\right)
    + \nabla p + \nabla p_c &= \phi{\bf e}_3,
    \\
    \label{eq:three-field_b}
    -\nabla\cdot{\bf u} +
    \nabla\cdot k\nabla p &=
    \nabla\cdot k{\bf e}_3,
    \\
    \label{eq:three-field_c}
    - \nabla\cdot {\bf u} - \zeta^{-1}p_c &= 0.
  \end{alignat}
\end{subequations}
The auxiliary variable $p_c$ is also known as the compaction pressure
(see~\citep{Katz:2007,Keller:2013}, for example).  Decomposing the
boundary of the domain by $\Gamma_{D} \cup \Gamma_N = \partial \Omega$
where $\Gamma_{D} \cap \Gamma_N = \emptyset$, and denoting the outward
unit normal vector on $\partial \Omega$ by ${\bf n}$, we consider the
following boundary conditions:
\begin{equation}
  \label{eq:bcs}
  \begin{split}
    {\bf u} &= {\bf g} \qquad \mbox{on}\ \Gamma_D,
    \\
    \eta{\bf D}{\bf u}\cdot{\bf n} - \left(\tfrac{1}{3}\eta\nabla\cdot{\bf u}
    + p + p_c\right){\bf n} &=  {\bf g}_N \quad\ \mbox{on} \ \Gamma_N,
    \\
      -k\left(\nabla p - {\bf e}_3\right)\cdot{\bf n} &= 0
      \qquad \mbox{on} \ \partial \Omega,
  \end{split}
\end{equation}
where ${\bf g} : \Gamma_D \to \mathbb{R}^d$ and ${\bf g}_N : \Gamma_N
\to \mathbb{R}^d$ are given boundary data. In the case $\partial
\Omega = \Gamma_D$, ${\bf g}$ is constructed to satisfy the
compatibility condition
\begin{equation}
  0 = \int_{\partial\Omega}{\bf g}\cdot{\bf n} \dif s.
\end{equation}
Note that the compatibility condition implies $\int_{\Omega}
\zeta^{-1} p_c \dif x = 0$ when $\Gamma_{D} = \partial \Omega$.

\section{Discrete formulation}
\label{s:weakform}

Assume $\Gamma_{D} = \partial \Omega$ and, without loss of generality,
homogeneous boundary conditions on ${\bf u}$. Define the function
space $L_0^2:=L_0^2(\Omega)=\{q\in L^2(\Omega) : \int_{\Omega}q\,dx =
0\}$ and let ${\bf X}_{h} \subset {\bf H}_{0}^{1}$ and $M_{h} \subset
(H^{1} \cap L_{0}^{2})$ be finite dimensional spaces. A mixed finite
element weak formulation for~\eqref{eq:three-field} is then given by:
find $({\bf u}_h, p_h, p_{ch}) \in {\bf X}_h \times M_h \times M_h$
such that
\begin{subequations}
  \label{eq:3-field-dwf}
  \begin{alignat}{3}
    \label{eq:3-field-dwf-a}
    a({\bf u}_h, {\bf v}) + b(p_h, {\bf v}) + b(p_{ch}, {\bf v}) &=
    \int_{\Omega}\phi {\bf e}_{3} \cdot {\bf v} \dif x
    \quad &&\forall {\bf v} \in {\bf X}_h,
    \\
    \label{eq:3-field-dwf-b}
    b(q, {\bf u}_h) - c(p_h, q)
    &= -\int_{\Omega} k{\bf e}_3 \cdot \nabla q \dif x \quad && \forall q\in M_h,
    \\
    \label{eq:3-field-dwf-c}
    b(\omega,{\bf u}_h) - d(p_{ch}, \omega) &= 0 \quad &&\forall \omega \in M_h,
  \end{alignat}
\end{subequations}
where
\begin{subequations}
  \label{eq:bilinForms}
  \begin{alignat}{1}
    \label{eq:bilinForms_a}
    a({\bf u},{\bf v})
    &=
    \int_{\Omega}\eta{\bf D}{\bf u}:{\bf D}{\bf v} \dif x
    -\int_{\Omega}\tfrac{1}{3}\eta(\nabla\cdot{\bf u})(\nabla\cdot{\bf v})
    \dif x,
    \\
    \label{eq:bilinForms_b}
    b(p,{\bf v}) &= -\int_{\Omega} p \nabla\cdot{\bf v} \dif x,
    \\
    \label{eq:bilinForms_c}
    c(p,q) &= \int_{\Omega}k\nabla p\cdot\nabla q \dif x,
    \\
    \label{eq:bilinForms_d}
    d(p,\omega) &= \int_{\Omega}\zeta^{-1} p\omega \dif x.
  \end{alignat}
\end{subequations}
We assume the choice of spaces ${\bf X}_h$ and $M_h$ satisfy the
inf-sup stability condition, but postpone the choice of the finite
element spaces until Section~\ref{s:numericalsimulations}.

Let $u \in \mathbb{R}^{n_{u}}$ be the vector of discrete velocity with
respect to the basis for ${\bf X}_h$, and let $p \in N^{n_p} = \{q \in
\mathbb{R}^{n_p} | q\ne 1\}$ be the vector of the discrete pressure
and $p_c\in N^{n_p}$ the vector of discrete compaction pressure, with
respect to the basis for~$M_h$. The discrete
system~\eqref{eq:3-field-dwf} can then be written in block matrix form
as
\begin{equation}
  \label{eq:blockMatrix}
  \begin{bmatrix}
    K_{\eta} & G^T & G^T \\
    G & -C_k & 0 \\
    G & 0 & -Q_{\zeta}
  \end{bmatrix}
  \begin{bmatrix}
    u \\ p \\ p_c
  \end{bmatrix} =
  \begin{bmatrix}
    f \\ g \\ 0
  \end{bmatrix},
\end{equation}
where $K_{\eta}$, $G$, $C_k$ and $Q_{\zeta}$ are the matrices obtained
from the discretization of the bi-linear forms $a(\cdot,\cdot)$,
$b(\cdot,\cdot)$, $c(\cdot,\cdot)$ and $d(\cdot,\cdot)$, respectively.
This is the system for which we wish to develop and deploy effective
preconditioners.

\section{Three-field block-preconditioners}
\label{s:preconditioning}

We now formulate and analyse block preconditioners for the system
in~\eqref{eq:blockMatrix}. To achieve this, we assume that $0 <
\eta,\zeta < \infty$, and $0 \le k < \infty$ are constants, in which
case \eqref{eq:blockMatrix} can be re-written as
\begin{equation}
  \label{eq:blockMatrix_c}
  \underbrace{
    \begin{bmatrix}
      \eta K & G^T & G^T \\
      G & -k C & 0 \\
      G & 0 & -\zeta^{-1}Q
    \end{bmatrix}
  }_{\mathcal{A}}
  \begin{bmatrix}
    u \\ p \\ p_c
  \end{bmatrix}
  =
  \begin{bmatrix}
    f \\ g \\ 0
  \end{bmatrix}.
\end{equation}
This format of the equations will guide us towards the correct scaling
of the different blocks in the preconditioners for the case of
non-constant $\eta$, $\zeta$ and~$k$.

To simplify the notation in the following, we introduce the shorthand
\begin{equation}
  \Bar{S} = G K^{-1} G^{T}.
\label{eq:g-bar}
\end{equation}
We assume that the spaces ${\bf X}_h$ and $M_h$ are chosen such that
\begin{equation}
  \ker G^{T} = \cbr{\bf 1},
\label{eq:discrete-inf-sup}
\end{equation}
where $\cbr{\bf 1}$ represents the arbitrary constant in the
pressure. In this case $\Bar{S}$ is invertible (since $K$ is positive
definite) on the space complementary to $\cbr{\bf 1}$. Finite element
spaces that are stable for Stokes equations
satisfy~\eqref{eq:discrete-inf-sup}. Indeed, so-called inf-sup stable
approximation spaces for Stokes have this property uniformly in
$h$. We note that $k$ and $\zeta$ are strictly positive and bounded,
hence the blocks $k C$ and $\zeta^{-1}Q$ are nonzero. However, if $k$
becomes small as $\zeta$ becomes large the second and third rows of
$\mathcal{A}$ will approach linear dependence. This degeneracy is
a modelling shortcoming of the considered equations.

For the proofs in this section, the following lemma will be used:
\begin{lemma}
  \label{lem:MN}
  Let $M$ and $N$ be symmetric and positive definite matrices. If $M -
  N$ is positive definite, then $N^{-1} - M^{-1}$ is positive
  definite.
\end{lemma}
\begin{proof}
  See \citet[Corollary 7.7.4]{Horn:book}.
\end{proof}

\subsection{Theoretical lower block triangular preconditioners}
\label{ss:blocktriangular_P}

We first consider lower block triangular preconditioners of the form
\begin{equation}
  \label{eq:LowPreconditioner}
  \mathcal{P} =
  \begin{bmatrix}
    \eta \mathcal{K} & 0 & 0 \\
    G & \mathcal{R} & 0 \\
    G & \mathcal{T} & \mathcal{S}
  \end{bmatrix}
\end{equation}
to precondition~\eqref{eq:blockMatrix_c}. Our objective is to find
expressions for $\mathcal{K}$, $\mathcal{R}$, $\mathcal{S}$ and
$\mathcal{T}$ such that the spectrum of the generalised eigenvalue
problem
\begin{equation}
  \label{eq:eigvalProb_low}
  \begin{bmatrix}
    \eta K & G^T & G^T \\
    G & -k C & 0 \\
    G & 0 & -\zeta^{-1}Q
  \end{bmatrix}
  \begin{bmatrix}
    u \\ p \\ p_c
  \end{bmatrix}
  = \Phi
  \begin{bmatrix}
    \eta \mathcal{K} & 0 & 0 \\
    G & \mathcal{R} & 0 \\
    G & \mathcal{T} & \mathcal{S}
  \end{bmatrix}
  \begin{bmatrix}
    u \\ p \\ p_c
  \end{bmatrix}
\end{equation}
is bounded independent of the mesh cell size $h$. In this case, the
iteration count for a Krylov method applied to the preconditioned
system
\begin{equation}
  \label{eq:preconsys}
  \begin{bmatrix}
    \eta \mathcal{K} & 0 & 0 \\
    G & \mathcal{R} & 0 \\
    G & \mathcal{T} & \mathcal{S}
  \end{bmatrix}^{-1}
  \begin{bmatrix}
    \eta K & G^T & G^T \\
    G & -k C & 0 \\
    G & 0 & -\zeta^{-1}Q
  \end{bmatrix}
  \begin{bmatrix}
    u \\ p \\ p_c
  \end{bmatrix} =
  \begin{bmatrix}
    \eta \mathcal{K} & 0 & 0 \\
    G & \mathcal{R} & 0 \\
    G & \mathcal{T} & \mathcal{S}
  \end{bmatrix}^{-1}
  \begin{bmatrix}
    f \\ g \\ 0
  \end{bmatrix}
\end{equation}
is expected to be optimal in terms of problem size.

The following theorem gives the conditions under which the problem
in~\eqref{eq:eigvalProb_low} admits only two distinct eigenvalues.
\begin{theorem}
  \label{thm:exactPrecon_low}
  Let the matrices $K$, $G$, $C$ and $Q$ and positive constants
  $\eta$, $\zeta$ and $k$ be those given
  in~\eqref{eq:blockMatrix_c}. In~\eqref{eq:LowPreconditioner}, if
  $\mathcal{K} = K$ and
  \begin{equation}
    \label{eq:RST_low-a}
    \mathcal{R}
    = -\frac{1}{\sigma\eta}GK^{-1}G^T + \frac{1}{\sigma\eta^2}GK^{-1}G^T
    \left(\frac{1}{\eta}GK^{-1}G^T+\frac{1}{\zeta}Q\right)^{-1}GK^{-1}G^T
    - \frac{k}{\sigma}C,
  \end{equation}
  \begin{equation}
    \label{eq:RST_low-b}
    \mathcal{S} = -\left(\frac{1}{\eta}GK^{-1}G^T+\zeta^{-1}Q\right),
  \end{equation}
  and
  \begin{equation}
    \label{eq:RST_low-c}
    \mathcal{T} = -\frac{1}{\eta}GK^{-1}G^T,
  \end{equation}
  where $\sigma$ is a parameter such that $\sigma < 0$ or
  $\sigma\in(0,1)$, then $\mathcal{R}$ is invertible and the
  generalised eigenvalue problem \eqref{eq:eigvalProb_low} has only
  two distinct eigenvalues, $\Phi_{1} = 1$ and $\Phi_{2} =
  \sigma$. Furthermore, the eigenvectors corresponding to the
  eigenvalue $\Phi_{1}=1$ have the form $\left[u^T \ 0 \ 0 \right]^T$.
\end{theorem}
\begin{proof}
  First we prove that $\mathcal{R}$ is invertible. For this, note that
  \begin{equation}
    q^T(\eta^{-2}\bar{S})^{-1}(\eta^{-1}\bar{S}+\zeta^{-1}Q)\bar{S}^{-1}q
    > q^T(\eta^{-2}\bar{S})^{-1}\eta^{-1}\bar{S}\bar{S}^{-1}q
    = q^T(\eta^{-1}\bar{S})^{-1}q,
  \end{equation}
  since $Q$ is positive definite. By Lemma~\ref{lem:MN} we therefore find that
  \begin{equation}
    q^T\eta^{-1}\bar{S}q >
    q^T\eta^{-1}\bar{S}(\eta^{-1}\bar{S}+\zeta^{-1}Q))^{-1}\eta^{-1}\bar{S}q.
  \end{equation}
  Since $kC$ is positive semi-definite it is easily seen from
  \eqref{eq:RST_low-a} that if $\sigma>0$, then $\mathcal{R}$ is
  negative definite and if $\sigma<0$, then $\mathcal{R}$ is positive
  definite. Hence, $\mathcal{R}$ is invertible.

  We now continue by proving that \eqref{eq:eigvalProb_low} has only
  two distinct eigenvalues. Assuming $\Phi = 1$,
  \eqref{eq:eigvalProb_low} becomes
  \begin{subequations}
    \label{eq:lambda1_low}
    \begin{alignat}{1}
      \label{eq:lambda1_low_a}
      \eta K u + G^T p + G^T p_c &= \eta K u \\
      \label{eq:lambda1_low_b}
      G u - k C p &= G u + \mathcal{R} p\\
      \label{eq:lambda1_low_c}
      G u - \zeta^{-1}Q p_c &= G u + \mathcal{T} p + \mathcal{S} p_c.
    \end{alignat}
  \end{subequations}
  From \eqref{eq:lambda1_low_a} we find $G^T(p + p_c) = 0$, hence~$p =
  - p_c$, provided both pressures have the same constant average.
  From~\eqref{eq:lambda1_low_b} we find that $(\mathcal{R} + kC) p =
  0$.  We need to show that $\mathcal{R} + kC = 0$ is non-singular, in
  which case~$p = 0$.  From the definition of~$\mathcal{R}$
  in~\eqref{eq:RST_low-a}
  \begin{equation}
    \mathcal{R} + kC = -\frac{1}{\sigma\eta}\mathcal{G}
    + \left(1 - \frac{1}{\sigma}\right) k C,
  \end{equation}
  where
  \begin{equation}
    \mathcal{G} = GK^{-1}G^T - GK^{-1}G^T\left(
      GK^{-1}G^T+\frac{\eta}{\zeta}Q\right)^{-1}GK^{-1}G^T.
  \end{equation}
  We now show that $\mathcal{G}$ is positive definite.  Using the
  shorthand from~\eqref{eq:g-bar} and defining $\Bar{Q} = \eta
  \zeta^{-1} Q$, we need to show that
  \begin{equation}
    q^T\left(\Bar{S} - \Bar{S}^T\left(\Bar{S}
        + \Bar{Q}\right)^{-1}\Bar{S}\right) q > 0
    \quad \forall q \in \mathbb{R}^{n_p},
  \end{equation}
  or equivalently
  \begin{equation}
    \label{eq:posdefinitematrix}
    \tilde{q}^T \Bar{S}^{-1} \tilde{q} - \tilde{q}^T\left(\Bar{S} +
      \Bar{Q}\right)^{-1} \tilde{q} > 0
    \quad \forall \tilde{q} \in \mathbb{R}^{n_p},
  \end{equation}
  where $\tilde{q} = \Bar{S} q$.  By Lemma~\ref{lem:MN}, since
  $\Bar{S} + \Bar{Q} - \Bar{S} = \Bar{Q}$ is positive definite
  (because $\eta \zeta^{-1} > 0$ and $Q$ is positive definite), the
  inequality in~\eqref{eq:posdefinitematrix} holds, hence
  $\mathcal{G}$ is positive definite.  If $\sigma < 0$, $\mathcal{R} +
  kC$ is positive definite (since $C$ is positive semi-definite), and
  if $\sigma \in (0,1)$ then $\mathcal{R} + kC$ is negative
  definite. It then follows that $p = p_c = 0$, and that $\Phi = 1$ is
  an eigenvalue of~\eqref{eq:eigvalProb_low} with eigenvector
  $\left[u^T \ 0 \ 0 \right]^T$.

  Next we assume $\Phi \ne 1$. Expanding the generalised eigenvalue
  problem~\eqref{eq:eigvalProb_low},
  \begin{subequations}
    \label{eq:lambdane1_low}
    \begin{alignat}{1}
      \label{eq:lambdane1_low_a}
      \eta K u + G^T p + G^T p_c &= \Phi\eta K u
      \\
      \label{eq:lambdane1_low_b}
      G u - k C p &= \Phi G u + \Phi \mathcal{R} p
      \\
      \label{eq:lambdane1_low_c}
      G u - \zeta^{-1}Q p_c &= \Phi G u + \Phi \mathcal{T} p
      + \Phi \mathcal{S} p_c.
    \end{alignat}
  \end{subequations}
  From~\eqref{eq:lambdane1_low_a},
  \begin{equation}
    \label{eq:Guexp_low}
    Gu = \frac{1}{\eta(\Phi - 1)} G K^{-1} G^T (p + p_c).
  \end{equation}
  Substituting this expression into~\eqref{eq:lambdane1_low_c} and
  using the definitions of $\mathcal{S}$~\eqref{eq:RST_low-b} and
  $\mathcal{T}$~\eqref{eq:RST_low-c} we find
  \begin{equation}
    (\Phi - 1) \left(\frac{1}{\eta} G K^{-1} G^T p
    + \left(\frac{1}{\eta} G K^{-1} G^T + \zeta^{-1} Q \right)
      p_c\right) = 0.
  \end{equation}
  Since $\Phi \ne 1$, it follows that
  \begin{equation}
    \label{eq:pc_low}
    p_c = -\frac{1}{\eta}\left(\frac{1}{\eta} G K^{-1} G^T
      + \zeta^{-1} Q \right)^{-1} G K^{-1} G^T p.
  \end{equation}
  Using \eqref{eq:Guexp_low} and~\eqref{eq:pc_low}
  in~\eqref{eq:lambdane1_low_b}, we find that
  \begin{equation}
    \mathcal{R} p = \left(-\frac{1}{\Phi\eta} \Bar{S}
      + \frac{1}{\Phi\eta^2} \Bar{S}
      \left(\frac{1}{\eta} \Bar{S} + \frac{1}{\zeta}Q\right)^{-1} \Bar{S}
      - \frac{k}{\Phi} C \right) p.
  \end{equation}
  From the definition of $\mathcal{R}$ in~\eqref{eq:RST_low-a} we
  have~$\Phi = \sigma$.
\end{proof}

While the choices for $\mathcal{K}$ $\mathcal{R}$, $\mathcal{S}$
and~$\mathcal{T}$ in Theorem~\ref{thm:exactPrecon_low} lead to the
generalised eigenvalue problem in~\eqref{eq:eigvalProb_low} having
only two distinct eigenvalues, it does not constitute a
computationally useful preconditioner. Computing the inverse of $G
K^{-1} G^T$ is not feasible for non-trivial problems.
For this reason, we consider in the next section a related, practical
preconditioner for~\eqref{eq:blockMatrix_c} for large scale
computations.
\subsection{Practical lower block triangular preconditioners}
\label{ss:pracblocktriangular_P}

Guided by the preconditioner developed in the previous section, we
proceed to formulate and analyse related preconditioners that are
practical for large-scale simulations. Our objective is to bound the
eigenvalues of the preconditioned system independently of the cell
size and, if possible, independently of the model parameters.

\subsubsection{Construction}

To construct a computationally feasible preconditioner, we need to
find a suitable approximation for the inverse of $G K^{-1} G^T$. For
this we make use of the following lemma.
\begin{lemma}
  \label{lem:spectralEquiv}
  The matrix $G K^{-1} G^T$ is spectrally equivalent to~$Q$:
  \begin{equation}
    \label{eq:specequivSchurQ}
    c_g \le \frac{\langle GK^{-1} G^Tq, q\rangle}{\langle Qq, q\rangle}
    \le c^g
  \end{equation}
  where $c_g$ and $c^g$ are positive constants independent of~$h$.
\end{lemma}
\begin{proof}
  See~\citet[Theorem~5.22]{Elman:book}.
\end{proof}

Lemma~\ref{lem:spectralEquiv} suggests that we may replace each
occurrence of $G K^{-1} G^T$ in the `theoretical' preconditioner by a
weighted pressure mass matrix $c_{i} Q$ in the expressions for
$\mathcal{R}$, $\mathcal{S}$ and $\mathcal{T}$ in
Theorem~\ref{thm:exactPrecon_low}, resulting in
\begin{equation}
  \label{eq:simpRST_lowbar-a}
  R = -\frac{1}{\sigma\eta}
  \left(c_1 - \frac{c_2 c_4}{c_3 + \eta \zeta^{-1}}\right)Q
    - \frac{k}{\sigma}C,
\end{equation}
\begin{equation}
  \label{eq:simpRST_lowbar-b}
  S = -\left(\frac{c_5}{\eta} + \zeta^{-1}\right) Q,
\end{equation}
and
\begin{equation}
  \label{eq:simpRST_lowbar-c}
  T = -\frac{c_6}{\eta} Q,
\end{equation}
respectively. Noting that $\eta \zeta^{-1}$ is positive, for an
admissible $\sigma$ (see Theorem~\ref{thm:exactPrecon_low}) we choose
to replace $R$ in~\eqref{eq:simpRST_lowbar-a} by a spectrally
equivalent operator
\begin{equation}
  \label{eq:simpRST_low-a}
  R = -\eta^{-1}Q - kC,
\end{equation}
Similarly, we choose to replace $S$ in~\eqref{eq:simpRST_lowbar-b} by
the spectrally equivalent
\begin{equation}
  \label{eq:simpRST_low-b}
  S = -\left((2\eta)^{-1} + \zeta^{-1}\right) Q
\end{equation}
(the above factor of two is based on computational experience).  We
choose to set $\mathcal{T} = T = 0$, which simplifies the
preconditioner.  We will show, in Section~\ref{sss:analysis}, that
this simplifies the analysis, without giving up bounds on the spectrum
of the preconditioned operator.

We now define a preconditioner for~\eqref{eq:blockMatrix_c} of the
form
\begin{equation}
  \label{eq:LowPreconditioner_T0}
  \mathcal{P}_t
  =
  \begin{bmatrix}
    \eta \mathcal{K} & 0 & 0 \\
    G & \mathcal{R} & 0 \\
    G & 0 & \mathcal{S}
  \end{bmatrix},
\end{equation}
in which the matrices $\mathcal{K}$, $\mathcal{R}$, and $\mathcal{S}$
satisfy
\begin{equation}
  \label{eq:specEquivInverse}
  c_k
  \le \frac{\langle Kq, q\rangle}{\langle \mathcal{K} q, q\rangle}
  \le c^k,
  \quad
  c_r
  \le
  \frac{\langle R q, q \rangle}{\langle \mathcal{R} q, q\rangle}
  \le c^r,
  \quad
  c_s
  \le
  \frac{\langle S q, q\rangle}{\langle \mathcal{S} q, q\rangle}
  \le c^s
\end{equation}
for $R$ in \eqref{eq:simpRST_low-a} and $S$
in~\eqref{eq:simpRST_low-b}, and where $c_{i}$ and $c^{i}$ in the
above are positive constants that are independent of $h$, $k$, $\eta$
and~$\zeta$.  In Section~\ref{s:numericalsimulations} we will consider
a preconditioner of the form in \eqref{eq:LowPreconditioner_T0} in
which $\mathcal{K} = K$, $\mathcal{R} = R$ and $\mathcal{S} = S$, with
the action of the inverse computed exactly via LU decomposition.  We
will denote this preconditioner by $\mathcal{P}^{{\rm LU}}_t$.  We
introduce~\eqref{eq:specEquivInverse} into the definition of the
preconditioner to permit a wider range of possible preconditioners
that can be computationally more efficient.  For example, to build an
efficient and scalable preconditioner, we consider in
Section~\ref{s:numericalsimulations} an approximation of inverses of
$K$, $R$ and $S$ by algebraic multigrid cycles, in which $\mathcal{K}
= K^{{\rm AMG}}$, $\mathcal{R} = R^{{\rm AMG}}$ and $\mathcal{S} =
S^{{\rm AMG}}$.  We will denote this preconditioner
by~$\mathcal{P}^{{\rm AMG}}_t$.  Multigrid approximations of $K$, $R$
and $S$ are spectrally equivalent approximations for the matrices in
question~\citep[Lemma 6.12]{Elman:book} and hence
satisfy~\eqref{eq:specEquivInverse}.

\subsubsection{Analysis}
\label{sss:analysis}

In proposing a practical preconditioner, we have thus far relied on
spectrally equivalent sub-matrices for guidance. We now prove that the
spectrum of the system of interest, preconditioned
by~\eqref{eq:LowPreconditioner_T0}, where $\mathcal{K}$, $\mathcal{R}$
and $\mathcal{S}$ satisfy~\eqref{eq:specEquivInverse}, can be bounded
independently of~$h$.  \citet{Klawonn:1998} proved eigenvalue bounds
for $2 \times 2$ block-triangular preconditioners for a class of
saddle point problems.  We follow a similar approach to
\citet{Klawonn:1998}, but generalised for $3 \times 3$
block-triangular preconditioners.

In the following we assume that $c_k > 1$
in~\eqref{eq:specEquivInverse}, and hence $K - \mathcal{K}$ is positive
definite. This is always possible by appropriate scaling even though
$\mathcal{K}=K$ would seem to be the simplest choice. We use this
assumption in the analysis; the choice $\mathcal{K}=K$ would lead to
significant degeneracy. However, no rescaling is necessary in the
numerical simulations in Section~\ref{s:numericalsimulations}. In
preparation for the analysis, we introduce some definitions.  Let
$\mathcal{A}$ be defined by~\eqref{eq:blockMatrix_c}
and~$\mathcal{P}_t$ by~\eqref{eq:LowPreconditioner_T0}, then
\begin{equation}
  \mathcal{P}_t^{-1}\mathcal{A}
  =
  \begin{bmatrix}
    \mathcal{K}^{-1}K & \eta^{-1} \mathcal{K}^{-1}G^T
    & \eta^{-1} \mathcal{K}^{-1}G^T \\
    -\mathcal{R}^{-1} G\mathcal{K}^{-1}(K-\mathcal{K}) &
    -\mathcal{R}^{-1} (\eta^{-1} \Tilde{S} + kC) &
    -\mathcal{R}^{-1}\eta^{-1} \Tilde{S} \\
    -\mathcal{S}^{-1}G\mathcal{K}^{-1}(K-\mathcal{K}) &
    -\mathcal{S}^{-1}\eta^{-1} \Tilde{S} &
    -\mathcal{S}^{-1}(\eta^{-1} \Tilde{S} + \zeta^{-1}Q)
  \end{bmatrix},
\end{equation}
where we have used the shorthand
\begin{equation}
  \Tilde{S} = G \mathcal{K}^{-1} G^{T}.
\end{equation}
Introducing
\begin{equation}
  \label{eq:compact_H}
  \mathcal{H} =
  \begin{bmatrix}
    \eta(K - \mathcal{K}) & 0 & 0 \\
    0 & -\mathcal{R} & 0 \\
    0 & 0 & -\mathcal{S}
  \end{bmatrix},
\end{equation}
we note that
\begin{equation}
  \label{eq:compact_HPA}
  \mathcal{H}\mathcal{P}_t^{-1}\mathcal{A} =
  \begin{bmatrix}
    \eta(K-\mathcal{K})\mathcal{K}^{-1}K & (K-\mathcal{K})\mathcal{K}^{-1}G^T &
    (K-\mathcal{K})\mathcal{K}^{-1}G^T
    \\
    G\mathcal{K}^{-1}(K-\mathcal{K}) &
    \eta^{-1}G\mathcal{K}^{-1}G^T + kC &
    \eta^{-1}G\mathcal{K}^{-1}G^T
    \\
    G\mathcal{K}^{-1}(K-\mathcal{K}) &
    \eta^{-1}G\mathcal{K}^{-1}G^T &
    \eta^{-1}G\mathcal{K}^{-1}G^T+\zeta^{-1}Q
  \end{bmatrix}.
\end{equation}
We also introduce
\begin{equation}
  \label{eq:compact_tildeH}
  \tilde{\mathcal{H}}
  =
  \begin{bmatrix}
    \eta K & 0 & 0
    \\
    0 & \eta^{-1} \Bar{S} +kC & 0
    \\
    0 & 0 & \eta^{-1} \Bar{S} + \zeta^{-1}Q
    - \eta^{-1} \Bar{S} (\eta^{-1} \Bar{S} +kC)^{-1}
    \eta^{-1} \Bar{S}
  \end{bmatrix}.
\end{equation}
We will consider bounds for $\mathcal{H} \mathcal{P}_{t}^{-1}
\mathcal{A}$ with respect to $\mathcal{H}$ (see also Lemma 3.4 of
\citet{Klawonn:1998}). To find these bounds, we first formulate some
intermediate results.  We use the notation~$A \le B$ to denote that $B
- A$~is symmetric positive semi-definite.
\begin{lemma}
  \label{lem:Dbounded_by_tildeH}
  Decomposing $\mathcal{H}\mathcal{P}_t^{-1}\mathcal{A}$ as
  \begin{equation}
    \label{eq:compact_hpatdt}
    \mathcal{H}\mathcal{P}_t^{-1}\mathcal{A}
    = \mathcal{L}\mathcal{D}\mathcal{L}^T,
  \end{equation}
  where
  \begin{equation}
    \mathcal{L}
    =
    \begin{bmatrix}
      I & 0 & 0
      \\
      \eta^{-1}GK^{-1} & I & 0
      \\
      \eta^{-1}GK^{-1} & \eta^{-1}GK^{-1}G^T(\eta^{-1}GK^{-1}G^T+kC)^{-1} & I
    \end{bmatrix}
  \end{equation}
  and
  \begin{equation}
    \mathcal{D}
    =
    \begin{bmatrix}
      \eta(K\mathcal{K}^{-1}K - K) & 0 & 0
      \\
      0 & \eta^{-1} \Bar{S} +kC & 0
      \\
      0 & 0 & \eta^{-1} \Bar{S} +\zeta^{-1}Q
      -  \eta^{-1} \Bar{S} (\eta^{-1} \Bar{S} +kC)^{-1} \eta^{-1} \Bar{S}
    \end{bmatrix},
  \end{equation}
  there exist positive constants $\hat{C}_0$, $\hat{C}_1$, independent
  of $h$, $k$, $\eta$ and $\zeta$ such that
  \begin{equation}
    \label{eq:compact_specequiv_H_tH}
    \hat{C}_0\tilde{\mathcal{H}} \le \mathcal{D}
    \le \hat{C}_1\tilde{\mathcal{H}}.
  \end{equation}
\end{lemma}
\begin{proof}
  From \eqref{eq:specEquivInverse} and~$c_k > 1$ it immediately
  follows that $\hat{C}_0 = \min\{(c_k-1), 1\}$ and $\hat{C}_1 =
  \max\{(c^k-1), 1\}$.
\end{proof}

\begin{lemma}
  \label{lem:L_tildeH_LT_bounds}
  Assume $\eta$ and $\zeta$ are positive bounded constants and defining
  \begin{equation}
    \label{eq:beta1def}
    \beta_1 = \frac{\eta\zeta^{-1}}{1 + \eta\zeta^{-1}},
  \end{equation}
  the eigenvalues of $\mathcal{L}\tilde{\mathcal{H}}\mathcal{L}^T$ are
  bounded by the extreme eigenvalues of $\tilde{\mathcal{H}}$:
  \begin{equation}
    \label{eq:compact_eigenvalues_bounded}
    \frac{1}{
      \max\left(4, \frac{6}{\beta_1\min\left(\frac{1}{c^g},1\right)}\right)}
    \tilde{\mathcal{H}} \le
    \mathcal{L}\tilde{\mathcal{H}}\mathcal{L}^T
    \le 5\max\left(1, \frac{1}{\beta_1\min\left(\frac{1}{c^g},
      1\right)}\right)\tilde{\mathcal{H}},
  \end{equation}
  where $c^g$ is given by Lemma~\ref{lem:spectralEquiv}.
\end{lemma}
\begin{proof}
  From
  \begin{equation}
    \mathcal{L}\tilde{\mathcal{H}}\mathcal{L}^T
    =
    \begin{bmatrix}
      \eta K & G^T & G^T \\
      G & 2\eta^{-1}\bar{S} + kC & 2\eta^{-1}\bar{S} \\
      G & 2\eta^{-1}\bar{S} & 2\eta^{-1}\bar{S} + \zeta^{-1}Q
    \end{bmatrix}
  \end{equation}
  we obtain
  \begin{equation}
    \label{eq:compact_xTHTx}
    \begin{split}
      x^T\mathcal{L}\tilde{\mathcal{H}}\mathcal{L}^Tx &=
      \begin{bmatrix}
        u \\ p \\ p_c
      \end{bmatrix}^T
      \mathcal{L}\tilde{\mathcal{H}}\mathcal{L}^T
      \begin{bmatrix}
        u \\ p \\ p_c
      \end{bmatrix}
      \\
      &\le \eta u^TKu + 2|p^TGu| + 2|p_c^TGu| + 4\eta^{-1}|p_c^T\bar{S}p|
      \\
      &\quad + p^T\left(2\eta^{-1}\bar{S}+kC\right)p
      + p_c^T\left(2\eta^{-1}\bar{S}+\zeta^{-1}Q\right)p_c.
    \end{split}
  \end{equation}
  Applying the Cauchy--Schwarz inequality and Young's inequality $ab
  \le a^2/2 + b^2/2$, we find
  \begin{equation}
    \label{eq:compact_ppcGup}
    \begin{split}
      |p^TGu| &\le \tfrac{1}{2}\left(p^T\left(\eta^{-1}\bar{S} + kC\right)p
      + u^T\eta K u\right),
      \\
      |p_c^TGu| &\le \tfrac{1}{2}\left(p_c^T\left(\eta^{-1}\bar{S}
      + \zeta^{-1}Q\right)p_c
    + u^T\eta K u\right),
      \\
      |p_c^T\bar{S}p| &\le \tfrac{1}{2}\left(p_c^T\bar{S}p_c
      + p^T\bar{S}p\right),
    \end{split}
  \end{equation}
  so that combining \eqref{eq:compact_xTHTx}
  and~\eqref{eq:compact_ppcGup} we obtain
  \begin{equation}
    \label{eq:compact_upperbound_tht_pre}
    \begin{split}
      x^T\mathcal{L}\tilde{\mathcal{H}}\mathcal{L}^Tx \le &
      3\eta u^TKu + 5 p^T\left(\eta^{-1}\bar{S} + kC\right)p \\
      &+ 5 p_c^T\left(\eta^{-1}\bar{S} + \zeta^{-1}Q\right)p_c.
    \end{split}
  \end{equation}
  From the definition of $\beta_1$ \eqref{eq:beta1def} and using
  Lemmas~\ref{lem:MN} and~\ref{lem:spectralEquiv}, we find
  \begin{equation}
    \label{eq:relating_H33}
    \begin{split}
      p_c^T\tilde{\mathcal{H}}_{33}p_c &=
      p_c^T\left(\eta^{-1}\bar{S} + \zeta^{-1}Q
        - \eta^{-1}\bar{S}(\eta^{-1} \bar{S} + kC)^{-1}\eta^{-1}\bar{S}\right)p_c \\
      &\ge p_c^T\zeta^{-1}Q p_c \ge
      p_c^T\beta_1\min\left(\frac{1}{c^g},1\right)\left(\eta^{-1}\bar{S}
      + \zeta^{-1}Q\right)p_c
    \end{split}
  \end{equation}
  so that
  \begin{equation}
    \label{eq:beta1pcterm}
    5p_c^T\left(\eta^{-1}\bar{S} + \zeta^{-1}Q\right)p_c \le
    \frac{5}{\beta_1\min\left(\frac{1}{c^g},1\right)}
    p_c^T\tilde{\mathcal{H}}_{33}p_c.
  \end{equation}
  Combining~\eqref{eq:compact_upperbound_tht_pre}
  and~\eqref{eq:beta1pcterm} we find the upper bound
  in~\eqref{eq:compact_eigenvalues_bounded}:
  \begin{equation}
    \label{eq:compact_upperbound_tht}
    \begin{split}
    x^T\mathcal{L}\tilde{\mathcal{H}}\mathcal{L}^Tx \le &
    3u^T\tilde{\mathcal{H}}_{11}u + 5p^T\tilde{\mathcal{H}}_{22}p
    +     \frac{5}{\beta_1\min\left(\frac{1}{c^g},1\right)}
    p_c^T\tilde{\mathcal{H}}_{33}p_c \\
    \le &
    5\max\left(1, \frac{1}{\beta_1\min\left(\frac{1}{c^g},1\right)}\right)
    x^T\tilde{\mathcal{H}}x.
    \end{split}
  \end{equation}

  To obtain the lower bound in~\eqref{eq:compact_eigenvalues_bounded}
  we follow \citet{Klawonn:1998} and consider
  \begin{equation}
    \label{eq:Rayleigh-quotient}
    \frac{x^T\mathcal{L} \tilde{\mathcal{H}}
      \mathcal{L}^Tx}{x^T\tilde{\mathcal{H}}x}
    = \frac{y^T\tilde{\mathcal{H}}y}{y^T\mathcal{L}^{-1}
      \tilde{\mathcal{H}}\mathcal{L}^{-T}y},
  \end{equation}
  where $y := \mathcal{L}^Tx$ was used as substitution. Now,
  \begin{equation}
    \mathcal{L}^{-1}\tilde{\mathcal{H}}\mathcal{L}^{-T} =
    \begin{bmatrix}
      \eta K & -G^T & -G^T + \Lambda^T \\
      -G & 2\eta^{-1}\bar{S}+kC & \Xi^T \\
      -G + \Lambda & \Xi & \Upsilon
    \end{bmatrix}
  \end{equation}
  where
  \begin{equation}
    \begin{split}
      \Lambda &= \eta^{-1}\bar{S}(\eta^{-1}\bar{S}+kC)^{-1}G
      \\
      \Xi &= -\eta^{-1}\bar{S}(\eta^{-1}\bar{S}+kC)^{-1}\eta^{-1}\bar{S}
      \\
      \Upsilon &= 2\eta^{-1}\bar{S} + \zeta^{-1}Q
      -2\eta^{-1}\bar{S}\left(\eta^{-1}\bar{S}+kC\right)^{-1}\eta^{-1}\bar{S}
      \\
      & \quad + \eta^{-1}\bar{S}\left(\eta^{-1}\bar{S}+kC\right)^{-1}\eta^{-1}\bar{S}(\eta^{-1}\bar{S}+kC)^{-1}\eta^{-1}\bar{S}.
    \end{split}
  \end{equation}
  Similar to the case of the upper bound, using the Cauchy--Schwarz
  inequality, Young's inequality and Lemma~\ref{lem:MN} it can be
  shown that
  \begin{equation}
    \label{eq:compact_ytTHTy_f_pre}
    \begin{split}
      y^T\mathcal{L}^{-1}\tilde{\mathcal{H}}\mathcal{L}^{-T}y
      &=
      \begin{bmatrix}
        v \\ q \\ q_c
      \end{bmatrix}^T
      \mathcal{L}^{-1}\tilde{\mathcal{H}}\mathcal{L}^{-T}
      \begin{bmatrix}
        v \\ q \\ q_c
      \end{bmatrix}
\\
      &\le 4\eta v^TKv
      + 4q^T\left(\eta^{-1}\bar{S}+kC\right)q
      + 6q_c^T\left(\eta^{-1}\bar{S} + \zeta^{-1}Q\right)q_c.
    \end{split}
  \end{equation}
  Using \eqref{eq:relating_H33} we note that
  \begin{equation}
    \label{eq:beta1lower}
    6q_c^T\left(\eta^{-1}\bar{S} + \zeta^{-1}Q\right)q_c
    \le
    \frac{6}{\beta_1\min\left(\frac{1}{c^g},1\right)}
    q_c^T\tilde{\mathcal{H}}_{33}q_c.
  \end{equation}
  Combining \eqref{eq:compact_ytTHTy_f_pre} and~\eqref{eq:beta1lower}
  we obtain
  \begin{equation}
    \begin{split}
      y^T\mathcal{L}^{-1}\tilde{\mathcal{H}}\mathcal{L}^{-T}y
      \le & 4 v^T\tilde{\mathcal{H}}_{11}v
      + 4q^T\tilde{\mathcal{H}}_{22}q
      +\frac{6}{\beta_1\min\left(\frac{1}{c^g},1\right)}
      q_c^T\tilde{\mathcal{H}}_{33}q_c \\
      \le &
      \max\left(4, \frac{6}{\beta_1\min\left(\frac{1}{c^g},1\right)}\right)
      y^T\tilde{\mathcal{H}}y.
    \end{split}
  \end{equation}
  Using~\eqref{eq:Rayleigh-quotient}
  \begin{equation}
    \label{eq:compact_ytTHTy_f_inx}
    x^T\mathcal{L}\tilde{\mathcal{H}}\mathcal{L}^{T}x
    \ge
    \frac{1}{
    \max\left(4, \frac{6}{\beta_1\min\left(\frac{1}{c^g},1\right)}\right)}
    x^T\tilde{\mathcal{H}}x,
  \end{equation}
  which is the lower bound
  in~\eqref{eq:compact_eigenvalues_bounded}.
\end{proof}

\begin{lemma}
  \label{lem:HPAbounds_approx}
  Assume $\eta$ and $\zeta$ are positive bounded constants and let
  $\beta_1$ be given by \eqref{eq:beta1def}. There exist positive
  constants $\tilde{C}_0$, $\tilde{C}_1$, independent of $h$ such that
  \begin{equation}
    \tilde{C}_0 \tilde{\mathcal{H}}
    \le \mathcal{H}\mathcal{P}_t^{-1}\mathcal{A} \le
    \tilde{C}_1 \tilde{\mathcal{H}},
  \end{equation}
where
  \begin{equation}
    \tilde{C}_0 = \frac{\hat{C}_0}{
      \max\left(4, \frac{6}{\beta_1\min\left(\frac{1}{c^g},1\right)}\right)}
    ,\quad
    \tilde{C}_1 = 5\max\left(1, \frac{1}{\beta_1\min\left(\frac{1}{c^g},1\right)}\right)\hat{C}_1
  \end{equation}
  and $\hat{C}_0$ and $\hat{C}_1$ are the constants in
  Lemma~\ref{lem:Dbounded_by_tildeH}.
\end{lemma}
\begin{proof}
  Combine~\eqref{eq:compact_specequiv_H_tH},
  \eqref{eq:compact_eigenvalues_bounded} and~\eqref{eq:compact_hpatdt}
  to find
  \begin{equation}
    \frac{\hat{C}_0}{
      \max\left(4, \frac{6}{\beta_1\min\left(\frac{1}{c^g},1\right)}\right)}
    \tilde{\mathcal{H}} \le
    \hat{C}_0\mathcal{L}\tilde{\mathcal{H}}\mathcal{L}^T \le
    \mathcal{H}\mathcal{P}_t^{-1}\mathcal{A}
  \end{equation}
  and
  \begin{equation}
    \mathcal{H}\mathcal{P}_t^{-1}\mathcal{A}
    \le \hat{C}_1\mathcal{L}\tilde{H}\mathcal{L}^T
    \le 5\max\left(1, \frac{1}{\beta_1\min\left(\frac{1}{c^g},1\right)}\right)\hat{C}_1\tilde{\mathcal{H}},
  \end{equation}
  from which the Lemma follows with
  \begin{equation}
    \tilde{C}_0 = \frac{\hat{C}_0}{
      \max\left(4, \frac{6}{\beta_1\min\left(\frac{1}{c^g},1\right)}\right)},
    \qquad
    \tilde{C}_1 = 5\max\left(1, \frac{1}{\beta_1\min\left(\frac{1}{c^g},1\right)}\right)\hat{C}_1.
  \end{equation}
\end{proof}

Building on Lemma~\ref{lem:HPAbounds_approx} we now formulate bounds
in terms of~$\mathcal{H}$.
\begin{lemma}
  \label{lem:HPAbounds}
  There exist positive constants $C_0$ and $C_1$, independent of $h$,
  such that
  \begin{equation}
    C_0 \mathcal{H} \le \mathcal{H}\mathcal{P}_t^{-1}\mathcal{A} \le
    C_1 \mathcal{H}.
  \end{equation}
\end{lemma}
\begin{proof}
  Building on Lemma~\ref{lem:HPAbounds_approx}, we need to show that
  $\mathcal{H}$~\eqref{eq:compact_H} is spectrally equivalent to
  $\tilde{\mathcal{H}}$~\eqref{eq:compact_tildeH}. We do so by showing
  spectral equivalence of the corresponding diagonal blocks in each
  matrix. It is clear that $\mathcal{H}_{11} = \eta(K - \mathcal{K})$
  is spectrally equivalent to $\tilde{\mathcal{H}}_{11} = \eta K$
  by~\eqref{eq:specEquivInverse}. Next, consider $\mathcal{H}_{22} =
  -\mathcal{R}$ and $\tilde{\mathcal{H}}_{22} = \eta^{-1}\bar{S} + kC$
  and find, using Lemma~\ref{lem:spectralEquiv}
  and~\eqref{eq:specEquivInverse},
  \begin{equation}
    \label{eq:compact_H2_upperlower}
    \begin{split}
      \tilde{\mathcal{H}}_{22} &= \eta^{-1}\bar{S}+kC
      \le -c_r\max(c^g,1)\mathcal{R}
      =  -c^{\mathcal{H}_2}\mathcal{R} \\
      \tilde{\mathcal{H}}_{22} &= \eta^{-1}\bar{S}+kC
      \ge -c^r\min(c_g,1)\mathcal{R}
      =  -c_{\mathcal{H}_{22}}\mathcal{R},
    \end{split}
  \end{equation}
  so that $-c_{\mathcal{H}_{22}}\mathcal{R} \le
  \tilde{\mathcal{H}}_{22} \le -c^{\mathcal{H}_{22}}\mathcal{R}$,
  meaning that $\tilde{\mathcal{H}}_{22}$ is spectrally equivalent to
  $\mathcal{H}_{22}$. Finally, consider $\mathcal{H}_{33} =
  -\mathcal{S}$ and $\tilde{\mathcal{H}}_{33} = \eta^{-1}\bar{S} +
  \zeta^{-1}Q - \eta^{-1}\bar{S}(\eta^{-1}\bar{S} +
  kC)^{-1}\eta^{-1}\bar{S}$. We note that, using
  Lemmas~\ref{lem:MN} and~\ref{lem:spectralEquiv}, and
  \eqref{eq:specEquivInverse},
  \begin{equation}
    \label{eq:compact_H3_upperlower}
    \begin{split}
      \tilde{\mathcal{H}}_{33} &= \eta^{-1}\bar{S}+\zeta^{-1}Q
      - \eta^{-1}\bar{S}(\eta^{-1}\bar{S} + kC)^{-1}\eta^{-1}\bar{S}
      \\
      \le & -2c_s\max(c^g,1)\mathcal{S} = -c^{\mathcal{H}_{33}}\mathcal{S}
      \\
      \tilde{\mathcal{H}}_{33} &= \eta^{-1}\bar{S}+\zeta^{-1}Q
      - \eta^{-1}\bar{S}(\eta^{-1}\bar{S} + kC)^{-1}\eta^{-1}\bar{S} \\
      \ge &\zeta^{-1}Q = \beta_2\left((2\eta)^{-1}+\zeta^{-1}\right)Q
      \ge  -c^s\beta_2\mathcal{S} = -c_{\mathcal{H}_{33}}\mathcal{S}
    \end{split}
  \end{equation}
  where
  \begin{equation}
    \beta_2 = \frac{\eta\zeta^{-1}}{\tfrac{1}{2} + \eta\zeta^{-1}}
  \end{equation}
  which is positive and bounded because $\eta$ and $\zeta$ are
  positive and bounded. From \eqref{eq:compact_H3_upperlower} it
  therefore follows that $-c_{\mathcal{H}_{33}}\mathcal{S} \le
  \tilde{\mathcal{H}}_{33} \le -c^{\mathcal{H}_{33}}\mathcal{S}$, hence
  $\tilde{\mathcal{H}}_{33}$ is spectrally equivalent to
  $\mathcal{S}$. Since $\tilde{\mathcal{H}}_{11}$,
  $\tilde{\mathcal{H}}_{22}$ and $\tilde{\mathcal{H}}_{33}$ are
  spectrally equivalent to $\eta(K-\mathcal{K})$, $\mathcal{R}$ and
  $\mathcal{S}$, respectively, $\mathcal{H}$ is spectrally equivalent
  to $\tilde{\mathcal{H}}$ and the Lemma follows.
\end{proof}

\begin{theorem}
  \label{thm:boundsspectrumPA}
  The spectrum of $\mathcal{P}_t^{-1}\mathcal{A}$ is bounded by the
  positive constants $C_0$ and $C_1$ from Lemma~\ref{lem:HPAbounds}:
  \begin{equation}
    \sigma\left(\mathcal{P}_t^{-1}\mathcal{A}\right)
    \subset \left[C_0, C_1\right].
  \end{equation}
\end{theorem}
\begin{proof}
  The proof can be found in~\citet[Theorem~3.5]{Klawonn:1998}, and is
  provided here for completeness. The constants $C_0$ and $C_1$ of
  Lemma~\ref{lem:HPAbounds} provide the lower and upper bounds for the
  eigenvalues of the generalized eigenvalue problem
  \begin{equation}
    \mathcal{H}\mathcal{P}_t^{-1}\mathcal{A} x = \lambda\mathcal{H}x.
  \end{equation}
  Since $\mathcal{H}$ is non-singular, the eigenvalues of the above
  problem are the same as the eigenvalues of
  \begin{equation}
    \mathcal{P}_t^{-1} \mathcal{A} y = \lambda y.
  \end{equation}
\end{proof}

We have thus far considered a lower block triangular preconditioner
$\mathcal{P}_t$~\eqref{eq:LowPreconditioner_T0}. An alternative would
be to consider an upper block triangular preconditioner
\begin{equation}
  \label{eq:compact_PU}
  \mathcal{P}_{tU}
  =
  \begin{bmatrix}
    \eta \mathcal{K} & G^T & G^T \\
    0 & \mathcal{R} & 0 \\
    0 & 0 & \mathcal{S}
  \end{bmatrix}.
\end{equation}
It was noted by \citet{Klawonn:1998} that since
\begin{equation}
  \mathcal{H} \mathcal{P}_t^{-1} \mathcal{A}
  = \mathcal{A}\mathcal{P}^{-1}_{tU}\mathcal{H},
\end{equation}
the results in this section also apply to the spectrum of~$\mathcal{A}
\mathcal{P}^{-1}_{tU}$.

\citet[Corollary 3.6]{Klawonn:1998} shows that $\mathcal{H}^{-1}$
defines an inner product on $\mathbb{R}^{n_u+2n_p}$ and that
$\mathcal{P}^{-1}\mathcal{A}$ is symmetric positive definite in this
inner product. This implies that one might employ the conjugate
gradient method in the $\mathcal{H}^{-1}$ inner product. Due to
Theorem~\ref{thm:boundsspectrumPA}, this then results in optimality in
$h$ of the preconditioned CG method. We, however, use the $\ell_2$
norm since using the $\mathcal{H}^{-1}$ inner product is not
computationally practical. A consequence of this choice is that we
will require Krylov methods for non-symmetric matrices. We consider
both GMRES and Bi-CGSTAB.

\citet[Theorem~3.7]{Klawonn:1998} provides bounds on the convergence
rate of the GMRES method using a norm equivalent to the
$\mathcal{H}^{-1}$-norm, while Theorem~3.8 of~\citep{Klawonn:1998}
provides bounds of the convergence rate of the GMRES method in the
$\ell_2$-norm. The convergence rate of the GMRES method in the
$\ell_2$-norm, unfortunately, depends on the condition number of the
matrix $\mathcal{H}^{1/2}$. This condition number may depend on $h$,
$\eta$, $\zeta$ and~$k$. Our numerical simulations in
Section~\ref{s:numericalsimulations}, however, do not show
problem-size dependence and show only slight dependence on the parameters
$\eta$, $\zeta$ and $k$. Problem size independence was also observed
in~\citep{Klawonn:1998}. In particular, the higher the
bulk-to-shear-viscosity ratio, the more iterations are required to
converge to a given tolerance. This is explained in
Lemma~\ref{lem:HPAbounds_approx} by noting that for small
$\eta\zeta^{-1}$, the constant $\beta_1$, defined by
\eqref{eq:beta1def}, tends to zero. This implies that $\tilde{C}_0$
and $\tilde{C}_1$ tend to, respectively, zero and infinity, leading to
an unbounded spectrum of $\mathcal{H}\mathcal{P}_t^{-1}\mathcal{A}$ in
the limit of large bulk-to-shear-viscosity ratio. The simulation
results seem less dependent on~$k$. Further discussion of the
parameters $k$, $\eta$ and $\zeta$ is deferred to
Section~\ref{s:numericalsimulations}.

\subsubsection{Variable parameter case}

Theorem~\ref{thm:boundsspectrumPA} states that the spectrum of
$\mathcal{P}_t^{-1}\mathcal{A}$, where $\mathcal{P}_t$ is the
preconditioner given by~\eqref{eq:LowPreconditioner_T0}, is bounded
above and below by constants independent of $h$. This indicates that
$\mathcal{P}_t$ will be a good preconditioner for the
system~\eqref{eq:blockMatrix_c}. To achieve this result we assumed
$k$, $\eta$ and $\zeta$ to be constant and that $\mathcal{K}$,
$\mathcal{R}$ and $\mathcal{S}$
satisfy~\eqref{eq:specEquivInverse}. For non-constant physical
parameters we propose to replace~$\eta K$ by~$K_{\eta}$ and set
\begin{equation}
  \label{eq:nonconstant_RS}
  R = -Q_{\eta} - C_k, \qquad
  S = -Q_{\eta}^{\zeta},
\end{equation}
which are such that they reduce to \eqref{eq:simpRST_low-a} and
\eqref{eq:simpRST_low-b}, respectively, when the physical parameters
are constant. Here $K_{\eta}$ and $C_k$ are the matrices defined in
Section~\ref{s:weakform} and $Q_{\eta}$ and $Q_{\eta}^{\zeta}$ are the
matrices obtained from the discretization of the bi-linear forms
$d_{\eta}(\cdot,\cdot)$ and $d_{\eta}^{\zeta}(\cdot,\cdot)$,
respectively, defined by
\begin{equation}
  \label{eq:detazeta}
  d_{\eta}(p,\omega)       = \int_{\Omega}\eta^{-1}p\omega \dif x, \qquad
  d_{\eta}^{\zeta}(p,\omega)
  = \int_{\Omega}\left((2\eta)^{-1}+\zeta^{-1}\right)p\omega \dif x.
\end{equation}

\section{Block-diagonal preconditioner}
\label{s:blockdiagonal_P}

The system in~\eqref{eq:blockMatrix_c} is symmetric and therefore the
use of MINRES would be allowed so long as the preconditioner is
symmetric and positive definite. The advantage of this over using
GMRES with the lower block triangular preconditioner of
Section~\ref{ss:blocktriangular_P} is that less memory is required;
the advantage over Bi-CGSTAB is that MINRES is guaranteed to
converge (subject to the usual floating-point caveats).  These
advantages motivate us to consider block-diagonal preconditioners
for~\eqref{eq:blockMatrix_c}.  In working toward a diagonal
preconditioner, ignoring the off-diagonal blocks
of~\eqref{eq:LowPreconditioner_T0} leads to a preconditioner of the
form
\begin{equation}
  \label{eq:diagPrecon_min}
  \mathcal{P}_{d}^{\star}
  =
  \begin{bmatrix}
    \eta \mathcal{K} & 0 & 0 \\
    0 & \mathcal{R} & 0 \\
    0 & 0 & \mathcal{S}
  \end{bmatrix},
\end{equation}
where $\mathcal{K}$, $\mathcal{R}$ and $\mathcal{S}$ are chosen such
that they satisfy~\eqref{eq:specEquivInverse}.  This preconditioner is
symmetric but not positive definite.  Multiplying
$\mathcal{P}_{d}^{\star}$ by the block diagonal matrix $\mathcal{J} =
{\rm bdiag}(I_u, -I_{p}, -I_{p})$, where $I_u \in \mathbb{R}^{n_u
  \times n_u}$ and $I_p \in \mathbb{R}^{n_p \times n_p}$ are identity
matrices, we obtain the symmetric positive definite preconditioner
\begin{equation}
  \label{eq:DPreconditioner}
  \mathcal{P}_d
  =
  \begin{bmatrix}
    \eta \mathcal{K} & 0 & 0 \\
    0 & -\mathcal{R} & 0 \\
    0 & 0 & -\mathcal{S}
  \end{bmatrix},
\end{equation}
which we propose for use with the MINRES method.

As in the previous section, for non-constant physical parameters
$\eta$, $\zeta$ and $k$, we propose to replace $\eta K$ by $K_{\eta}$
and we let $R$ and $S$ be given by~\eqref{eq:nonconstant_RS}.

By contrast to the block-triangular preconditioner, we have not been
able to prove boundedness of the spectrum for this block diagonal
preconditioner. Instead, in Section~\ref{s:numericalsimulations} we
investigate its properties by computation.

\section{Numerical simulations}
\label{s:numericalsimulations}

In this section we examine numerically the performance of the proposed
preconditioners. For both the block-triangular and block-diagonal
preconditioners we consider two approaches for the action of the
inverse of the matrices $\mathcal{K}$, $\mathcal{R}$ and
$\mathcal{S}$; LU decomposition (denoted by $\mathcal{P}^{LU}$) and
algebraic multigrid (AMG, denoted $\mathcal{P}^{{\rm AMG}}$).  We use
the LU decomposition as reference preconditioner to which the AMG
preconditioner can be compared. The LU-based preconditioner, however, is not
suitable for large-scale computations. We also remark that with the LU
preconditioner, the constants in \eqref{eq:specEquivInverse} are all
equal to unity. When using the AMG-based preconditioners, we use
smoothed aggregation algebraic multigrid for the $\mathcal{K}$ block
and classical algebraic multigrid for the $\mathcal{R}$ and
$\mathcal{S}$ blocks. We use a single multigrid V-cycle, and unless
otherwise stated, for smoothed aggregation we use four applications of
a Chebyshev smoother with one symmetric Gauss--Seidel iteration for
each Chebyshev application.  Although we use AMG to compute the action
of the inverse of the matrices $\mathcal{K}$, $\mathcal{R}$ and
$\mathcal{S}$, we note that any spectrally equivalent operator may be
used.

In all test cases we use $P^2$--$P^1$--$P^1$ continuous Lagrange
finite elements on simplices; this combination
satisfies~\eqref{eq:discrete-inf-sup} and will be inf--sup stable
since $\zeta^{-1} > 0$. More is needed when~$\zeta^{-1} = 0$. We
terminate the solver once a relative true residual of $10^{-8}$ is
reached.  In the case of GMRES, we use a restarted method with
restarts after $k$ iterations. We denote this by GMRES($k$).

All experiments have been performed using libraries from the FEniCS
Project~\citep{logg:2010,fenics:book} and the block preconditioning
support from PETSc~\citep{brown:2012}. For smoothed aggregation AMG,
we use the library ML~\citep{gee:2006} while classical algebraic
multigrid is used via the BoomerAMG library~\citep{henson:2002}. The
source code for reproducing all examples is freely available in the
supporting material~\citep{supporting_code}.

\subsection{Constant bulk and shear viscosity test case in two
  dimensions}
\label{ss:tc1}

In this test case we consider the simplified two-phase flow equations
in~\eqref{eq:mckenzie_old}. For the parameters
in~\eqref{eq:three-field}, we set $\eta = 1$, $\zeta = \alpha + 1/3$
and
\begin{multline}
  \label{eq:manu_k}
  k =
  \frac{k^{*} - k_{*}}{4\tanh(5)}
  \left(\tanh(10 x - 5)
    + \tanh(10z - 5) \right.
  \\
  \left.
    +  \frac{2(k^{*} - k_{*})
      - 2 \tanh(5)(k_{*}
      + k^{*})}{k_{*} - k^{*}}
    + 2\right),
\end{multline}
where $k_{*}$ and $k^{*}$ are parameters that control the maximum and
minimum values of $k$ is a domain. We ignore the buoyancy terms but
prescribe a source term $\mathbf{f}$ in~\eqref{eq:three-field_a}. The
Dirichlet boundary condition and the source term are constructed such
that the exact solution is
\begin{align}
  \label{eq:exact_ux}
  u_{x} &= k \partial_{x} p + \sin(\pi x)\sin(2\pi z) + 2,
  \\
  \label{eq:exact_uz}
  u_{z} &= k \partial_{z} p +  \frac{1}{2}\cos(\pi x)\cos(2\pi z) + 2.
  \\
  \label{eq:exact_p}
  p     &= -\cos(4 \pi x)\cos(2 \pi z),
\end{align}
We consider this test case on a structured triangular mesh of the unit
square $\Omega=[0, 1]^2$. This test case was studied
in~\citep{Rhebergen:2014}.

\subsubsection{Iteration counts for different preconditioners}

We first consider the case of $(k_*, k^*) = (0.5, 1.5)$ for the
diagonal and block triangular preconditioners.
Table~\ref{tab:tc1_its_alpha_minres} presents the iteration counts for
the two-field preconditioner from \citep{Rhebergen:2014} and the
three-field block diagonal preconditioner from this work, using
MINRES. In both cases, LU and AMG versions are considered. The results
in Table~\ref{tab:tc1_its_alpha_minres} show that the LU versions of
the preconditioners are optimal, with the three-field version showing
some sensitivity to the $\alpha$ parameter. Of greater interest is the
performance of the AMG-based preconditioners. In this case, the
iteration count of the three-field preconditioner is largely
insensitive to the problem size or the $\alpha$ parameter. For the
two-field AMG preconditioner, the iteration count has a strong
dependency on~$\alpha$. It is this observation
from~\citep{Rhebergen:2014} that motivated the present work.

\begin{table}
  \caption{Iteration counts for constant bulk and shear viscosity
    tests using the two- and three-field block diagonal preconditioners
    for different values of~$\alpha$. The
    number of cells in the mesh is~$2N$.}
\begin{center}
\begin{tabular}{c|cccccc}
 \multicolumn{7}{c}{$\mathcal{P}_{d2}^{LU}$/MINRES}\\
 \multicolumn{7}{c}{}\\
\hline
$N$ & $\alpha=-\tfrac{1}{3}$ & $\alpha=0$ & $\alpha=1$ & $\alpha=10$ & $\alpha=100$ & $\alpha=1000$ \\
\hline
 $32^2$& 8 & 8 & 8 & 7 & 7 & 5 \\
 $64^2$& 8 & 8 & 8 & 7 & 7 & 5 \\
$128^2$& 8 & 8 & 7 & 7 & 7 & 5 \\
$256^2$& 7 & 6 & 6 & 6 & 7 & 4 \\
\hline
 \multicolumn{7}{c}{}\\
 \multicolumn{7}{c}{$\mathcal{P}_{d3}^{LU}$/MINRES}\\
 \multicolumn{7}{c}{}\\
\hline
$N$ & $\alpha=-\tfrac{1}{3}$ & $\alpha=0$ & $\alpha=1$ & $\alpha=10$ & $\alpha=100$ & $\alpha=1000$ \\
\hline
 $32^2$& 8 & 15 & 22 & 33 & 39 & 39 \\
 $64^2$& 8 & 15 & 21 & 33 & 37 & 39 \\
$128^2$& 8 & 15 & 21 & 33 & 37 & 39 \\
$256^2$& 8 & 13 & 21 & 33 & 39 & 39 \\
\hline
 \multicolumn{7}{c}{}\\
 \multicolumn{7}{c}{$\mathcal{P}_{d2}^{AMG}$/MINRES}\\
 \multicolumn{7}{c}{}\\
\hline
$N$ & $\alpha=-\tfrac{1}{3}$ & $\alpha=0$ & $\alpha=1$ & $\alpha=10$ & $\alpha=100$ & $\alpha=1000$ \\
\hline
 $32^2$& 19 & 19 & 23 & 40 & 93  & 238 \\
 $64^2$& 23 & 25 & 29 & 48 & 115 & 289 \\
$128^2$& 26 & 29 & 33 & 57 & 133 & 338 \\
$256^2$& 30 & 32 & 36 & 66 & 155 & 388 \\
\hline
 \multicolumn{7}{c}{}\\
 \multicolumn{7}{c}{$\mathcal{P}_{d3}^{AMG}$/MINRES}\\
 \multicolumn{7}{c}{}\\
\hline
$N$ & $\alpha=-\tfrac{1}{3}$ & $\alpha=0$ & $\alpha=1$ & $\alpha=10$ & $\alpha=100$ & $\alpha=1000$ \\
\hline
 $32^2$& 29 & 25 & 35 & 60 & 73 & 82 \\
 $64^2$& 34 & 29 & 39 & 66 & 84 & 73 \\
$128^2$& 38 & 32 & 43 & 73 & 97 & 84 \\
$256^2$& 43 & 36 & 46 & 78 & 108& 95 \\
\hline
\end{tabular}
\label{tab:tc1_its_alpha_minres}
\end{center}
\end{table}

Tables~\ref{tab:tc1_its_alpha_bicgstab}
and~\ref{tab:tc1_its_alpha_gmres} present the number of iterations
required to converge for the block triangular preconditioner using
Bi-CGSTAB and GMRES(100), respectively. Compared to the three field
diagonal preconditioner, the triangular preconditioner requires fewer
iterations. For the AMG-based cases, the three-field preconditioner
requires two to four times fewer iterations than the two-field
preconditioner when $\alpha = 100$ and $\alpha = 1000$.  Noteworthy is
that the Bi-CGSTAB tests require fewer iterations than the GMRES(100)
tests. Overall, the three-field preconditioners show less sensitivity
to the parameter $\alpha$ than the two-field preconditioner
of~\citet{Rhebergen:2014} (see Table~\ref{tab:tc1_its_alpha_minres}).
\begin{table}
  \caption{Iteration counts for the constant bulk and shear viscosity
    tests using the three-field block triangular preconditioner and
    Bi-CGSTAB for different values of~$\alpha$. The number of cells
    in the mesh is~$2N$.}
\begin{center}
\begin{tabular}{c|cccccc}
 \multicolumn{7}{c}{$\mathcal{P}_t^{LU}$/Bi-CGSTAB}\\
 \multicolumn{7}{c}{}\\
\hline
$N$ & $\alpha=-\tfrac{1}{3}$ & $\alpha=0$ & $\alpha=1$ & $\alpha=10$ & $\alpha=100$ & $\alpha=1000$ \\
\hline
 $32^2$& 2 & 5 & 7 & 10 & 12 & 12 \\
 $64^2$& 2 & 4 & 7 & 11 & 13 & 13 \\
$128^2$& 2 & 4 & 7 & 11 & 13 & 13 \\
$256^2$& 2 & 4 & 7 & 11 & 13 & 14 \\
\hline
 \multicolumn{7}{c}{}\\
 \multicolumn{7}{c}{$\mathcal{P}_t^{AMG}$/Bi-CGSTAB}\\
 \multicolumn{7}{c}{}\\
\hline
$N$ & $\alpha=-\tfrac{1}{3}$ & $\alpha=0$ & $\alpha=1$ & $\alpha=10$ & $\alpha=100$ & $\alpha=1000$ \\
\hline
 $32^2$& 6  & 7 & 10 & 20 & 25 & 27 \\
 $64^2$& 8  & 8 & 12 & 22 & 28 & 34 \\
$128^2$& 10 & 9 & 12 & 23 & 36 & 41 \\
$256^2$& 11 & 10& 12 & 27 & 41 & 50 \\
\hline
\end{tabular}
\label{tab:tc1_its_alpha_bicgstab}
\end{center}
\end{table}

\begin{table}
  \caption{Iteration counts for the constant bulk and shear viscosity
    tests using the three-field block triangular preconditioner and
    GMRES(100) for different values of~$\alpha$. The number of cells in
    the mesh is~$2N$.}
\begin{center}
\begin{tabular}{c|cccccc}
 \multicolumn{7}{c}{$\mathcal{P}_t^{LU}$/GMRES(100)}\\
 \multicolumn{7}{c}{}\\
\hline
$N$ & $\alpha=-\tfrac{1}{3}$ & $\alpha=0$ & $\alpha=1$ & $\alpha=10$ & $\alpha=100$ & $\alpha=1000$ \\
\hline
 $32^2$& 4 & 8 & 12 & 19 & 21 & 21 \\
 $64^2$& 4 & 8 & 12 & 19 & 21 & 22 \\
$128^2$& 4 & 8 & 12 & 19 & 22 & 23 \\
$256^2$& 4 & 8 & 12 & 18 & 22 & 23 \\
\hline
 \multicolumn{7}{c}{}\\
 \multicolumn{7}{c}{$\mathcal{P}_t^{AMG}$/GMRES(100)}\\
 \multicolumn{7}{c}{}\\
\hline
$N$ & $\alpha=-\tfrac{1}{3}$ & $\alpha=0$ & $\alpha=1$ & $\alpha=10$ & $\alpha=100$ & $\alpha=1000$ \\
\hline
 $32^2$& 11 & 13 & 18 & 33 & 41 & 45 \\
 $64^2$& 13 & 14 & 21 & 39 & 50 & 54 \\
$128^2$& 15 & 16 & 24 & 42 & 58 & 48 \\
$256^2$& 18 & 18 & 27 & 48 & 66 & 58 \\
\hline
\end{tabular}
\label{tab:tc1_its_alpha_gmres}
\end{center}
\end{table}

\subsubsection{Observed convergence rates}

To understand the behaviour of the two- and three-field
preconditioners, it is helpful to examine the change in the
residual with iteration count. This is shown in
Figure~\ref{fig:tc1_alpha} for the AMG-based two- and three-field
block diagonal preconditioners with MINRES, and the block-triangular
preconditioner with Bi-CGSTAB and GMRES(100).  We observe that the
residual with the two-field preconditioner reduces rapidly to a
relative residual of approximately $10^{-4}$, at which point the
convergence slows. This behaviour is not observed with the three-field
preconditioners.

We note that dropping just three orders in magnitude in the residual
is a criterion for convergence that is commonly used. With a
relative tolerance of $10^{-3}$, the performance of the two field
preconditioner would appear to be very good and we would draw
substantially different conclusions on the relative merits of the two-
and three-field formulations.  However, we have performed tests that
show that a much tighter tolerance is required to maintain the
convergence rates to the exact solution with mesh refinement.  We
discuss this in Appendix~\ref{ap:femRates}.

\begin{figure}
  \centering
  \includegraphics[width=0.6\textwidth]{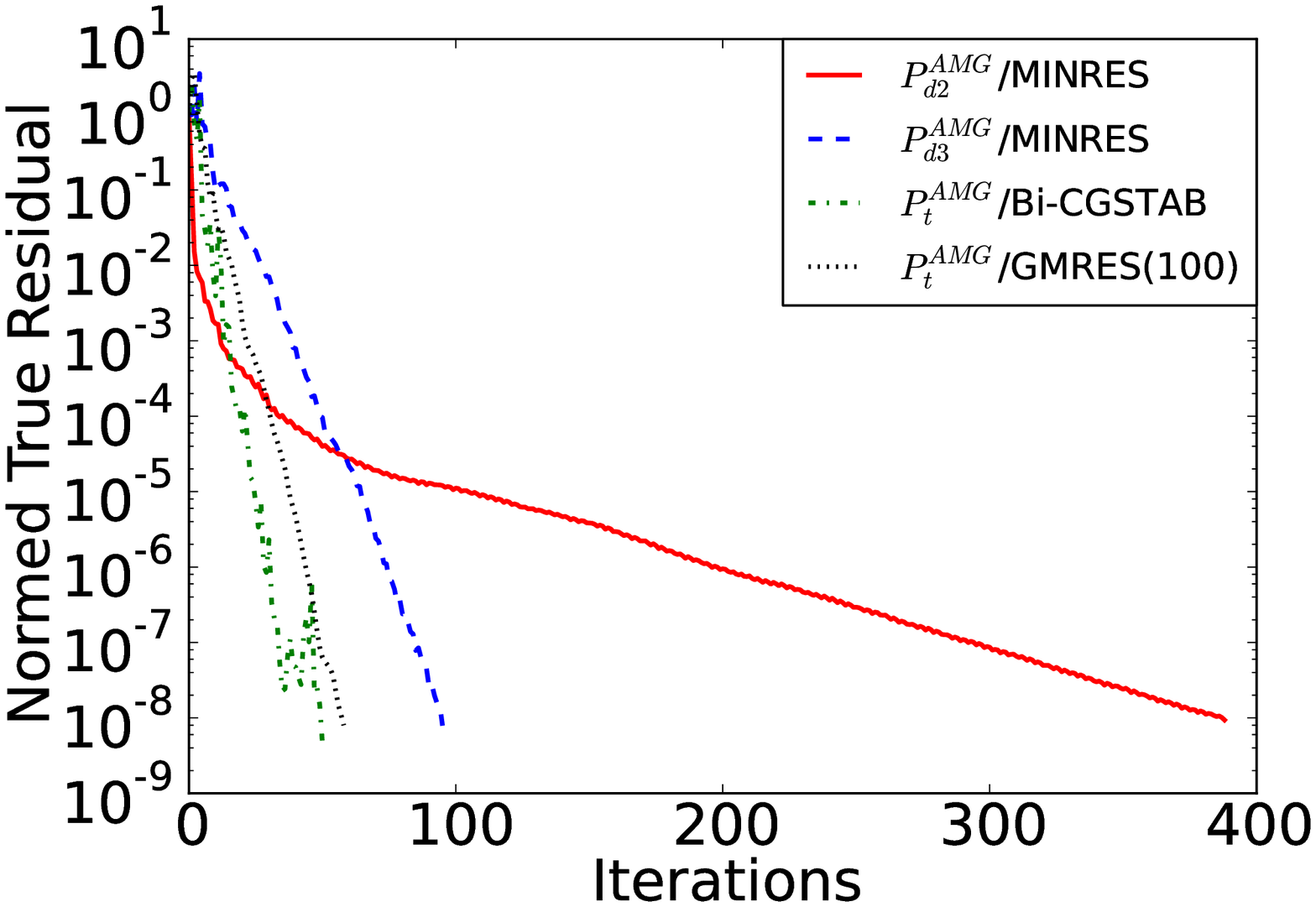}
  \caption{Residual decrease using preconditioned MINRES (with
    $\mathcal{P}_{d2}^{AMG}$ and $\mathcal{P}_{d3}^{AMG}$) and
    preconditioned Bi-CGSTAB and GMRES(100) (with
    $\mathcal{P}_t^{AMG}$) for the unit square test of
    Section~\ref{ss:tc1}. In all cases a mesh size of $N=256^2$ and
    $\alpha = 1000$ are used.}
  \label{fig:tc1_alpha}
\end{figure}

\subsection{Variable bulk and shear viscosity test case in two
  dimensions}
\label{ss:tc2}

In this test case we consider a manufactured solution for the prescribed
porosity field
\begin{equation*}
  \phi = \frac{1}{2}(\phi_*+\phi^*)
  + \frac{1}{2}(\phi^{*} - \phi_{*})\cos(4\pi(x\sin(\pi/6) + z\cos(\pi/6))),
\end{equation*}
where $\phi_{*}$ and $\phi^{*}$ are prescribed and $\phi_{*} \le \phi
\le \phi^{*}$. Let $\eta$, $\zeta$ and $k$ be given by
\begin{equation}
  \label{eq:constitutiveRel}
  k = \frac{R^2}{r_{\zeta} + 4/3}\left(\frac{\phi}{\phi_0}\right)^m, \quad
  \eta = 2\exp(-\lambda(\phi-\phi_0)), \quad
  \zeta = r_{\zeta}\left(\frac{\phi}{\phi_0}\right)^{-1},
\end{equation}
where $R=\delta/H$, with $\delta$ the reference compaction length
$\delta = \sqrt{(r_{\zeta} + 4/3) \eta_0 k_0 / \mu}$, $\mu$ is the
(constant) melt viscosity, $k_0$ is the characteristic permeability,
$\eta_0$ is the characteristic shear viscosity,
$r_{\zeta}=\zeta_0/\eta_0$ with $\zeta_0$ the characteristic bulk
viscosity and $H$ a length scale. Furthermore, $\phi_0$ is the
characteristic porosity and $m$ and $\lambda$ are constants. We choose
$m = 2$, $\lambda = 27$, $r_{\zeta} = 5/3$, $R=0.1$ and $\phi_0 =
0.05$. As before, we neglect buoyancy and add a source term
$\mathbf{f}$ to the right hand side of~\eqref{eq:three-field_a}. Again
the Dirichlet boundary condition and the source term are such that the
exact solution for the velocity $\mathbf{u}$ and the pressure $p$ are
given by~\eqref{eq:exact_ux}, \eqref{eq:exact_uz}
and~\eqref{eq:exact_p}. The approximate solution is computed on a
structured triangular mesh of the unit square, $\Omega=[0,1]^2$.  For
the tests in this section, we fix $\phi^{*} = 0.3$ and
vary~$\phi_{*}$.

To compare more fairly the three-field block diagonal preconditioner
$\mathcal{P}_{d3}$ with the two-field block diagonal preconditioner
$\mathcal{P}_{d2}$ introduced in \citep{Rhebergen:2014}, we slightly
modify the two-field block-preconditioner of~\citep{Rhebergen:2014} to
include a porosity dependence. This modification is described in
Appendix~\ref{ap:twofieldviscosityPreconditioner}.

In Table~\ref{tab:tc2_its_minres} we record the number of iterations
for the block diagonal preconditioners with MINRES for different
values of~$\phi_{*}$. For the AMG-based preconditioner, the iteration
count is lower for the three-field preconditioner. For low values of
$\phi_{*}$ we could not compute converged solutions with the two-field
preconditioner.
\begin{table}
  \caption{Iteration counts for the block diagonal preconditioners
    using MINRES for the variable viscosity test in two dimensions.
    The number of cells in the mesh is~$2N$. A dash indicates that we
    could not compute a converged solution to to breakdown of the
    solver.}
\begin{center}
\begin{tabular}{c|ccccccc}
\multicolumn{1}{c}{}& \multicolumn{3}{c}{$\mathcal{P}_{d2}^{LU}$/MINRES} & \multicolumn{3}{c}{$\mathcal{P}_{d2}^{AMG}$/MINRES}\\
 \cline{1-4} \cline{6-8}
$N$ & $\phi_*=10^{-3}$ & $\phi_*=10^{-5}$ & $\phi_*=0$ & & $\phi_*=10^{-3}$ & $\phi_*=10^{-5}$ & $\phi_*=0$ \\
 \cline{1-4} \cline{6-8}
 $32^2$& 76 & - & - & & 317 & - & - \\
 $64^2$& 71 & - & - & & 382 & - & - \\
$128^2$& 70 & - & - & & 428 & - & - \\
$256^2$& 67 & - & - & & 463 & - & - \\
 \cline{1-4} \cline{6-8}
 \multicolumn{8}{c}{}\\
\multicolumn{1}{c}{}& \multicolumn{3}{c}{$\mathcal{P}_{d3}^{LU}$/MINRES} & \multicolumn{3}{c}{$\mathcal{P}_{d3}^{AMG}$/MINRES}\\
 \cline{1-4} \cline{6-8}
$N$ & $\phi_*=10^{-3}$ & $\phi_*=10^{-5}$ & $\phi_*=0$ & & $\phi_*=10^{-3}$ & $\phi_*=10^{-5}$ & $\phi_*=0$ \\
 \cline{1-4} \cline{6-8}
 $32^2$& 217 & 218 & 219 & & 249 & 251 & 251 \\
 $64^2$& 223 & 226 & 226 & & 227 & 229 & 229 \\
$128^2$& 216 & 222 & 222 & & 227 & 243 & 243 \\
$256^2$& 213 & 244 & 247 & & 239 & 297 & 299 \\
 \cline{1-4} \cline{6-8}
\end{tabular}
\label{tab:tc2_its_minres}
\end{center}
\end{table}
Tables~\ref{tab:tc2_its_bicgstab} and~\ref{tab:tc2_its_gmres} record
the number of iterations for the block triangular preconditioner using
Bi-CGSTAB and GMRES(100), respectively. We observe minimal sensitivity
to~$\phi_{*}$ and, again, the iteration count for Bi-CGSTAB is
significantly lower than for GMRES(100). For both Bi-CGSTAB and
GMRES(100), the iteration counts with the block triangular
preconditioner are significantly less than for the two- and three-field
block-diagonal preconditioners with MINRES.
\begin{table}
  \caption{Iteration counts for the block-triangular preconditioner
    using Bi-CGSTAB for the variable viscosity test in two dimensions.
    The number of cells in the mesh is~$2N$.}
\begin{center}
\begin{tabular}{c|ccccccc}
\multicolumn{1}{c}{}& \multicolumn{3}{c}{$\mathcal{P}_{t}^{LU}$/Bi-CGSTAB} & \multicolumn{3}{c}{$\mathcal{P}_{t}^{AMG}$/Bi-CGSTAB}\\
 \cline{1-4} \cline{6-8}
$N$ & $\phi_*=10^{-3}$ & $\phi_*=10^{-5}$ & $\phi_*=0$ & & $\phi_*=10^{-3}$ & $\phi_*=10^{-5}$ & $\phi_*=0$ \\
 \cline{1-4} \cline{6-8}
 $32^2$& 72 & 75 & 68 & & 71 & 70 & 69 \\
 $64^2$& 63 & 69 & 65 & & 61 & 60 & 61 \\
$128^2$& 69 & 67 & 68 & & 51 & 52 & 51 \\
$256^2$& 69 & 61 & 54 & & 48 & 64 & 72 \\
 \cline{1-4} \cline{6-8}
\end{tabular}
\label{tab:tc2_its_bicgstab}
\end{center}
\end{table}

\begin{table}
  \caption{Iteration counts for the block-triangular preconditioner
    using GMRES(100) for the variable viscosity test in two dimensions.
    The number of cells in the mesh is~$2N$.}
\begin{center}
\begin{tabular}{c|ccccccc}
\multicolumn{1}{c}{}& \multicolumn{3}{c}{$\mathcal{P}_{t}^{LU}$/GMRES(100)} & \multicolumn{3}{c}{$\mathcal{P}_{t}^{AMG}$/GMRES(100)}\\
 \cline{1-4} \cline{6-8}
$N$ & $\phi_*=10^{-3}$ & $\phi_*=10^{-5}$ & $\phi_*=0$ & & $\phi_*=10^{-3}$ & $\phi_*=10^{-5}$ & $\phi_*=0$ \\
 \cline{1-4} \cline{6-8}
 $32^2$& 124 & 112 & 123 & & 121 & 115 & 115 \\
 $64^2$& 108 & 118 & 114 & & 92  & 94  & 94  \\
$128^2$& 113 & 104 & 98  & & 85  & 86  & 86  \\
$256^2$&  96 & 104 & 104 & & 81  & 102 & 102 \\
 \cline{1-4} \cline{6-8}
\end{tabular}
\label{tab:tc2_its_gmres}
\end{center}
\end{table}

\subsection{Constant bulk and shear viscosity test case in three
  dimensions}
\label{ss:tc3}

In this test case we consider a subduction zone setting with a
geometry described in Figure~\ref{fig:3Dsubduction} (this is the same
geometry used in~\citep{Rhebergen:2014}). Boundary sub-domains are
defined as $\Gamma_1 = \{{\bf x} \, | \, x + z = 1\}$, $\Gamma_2 =
\{{\bf x} \, | \, z = 1\}$, $\Gamma_3 = \partial \Omega
\backslash(\Gamma_{1} \cup \Gamma_{2})$. In this test case we solve
the simplified two-phase flow equations
in~\eqref{eq:mckenzie_old}. For the parameters
in~\eqref{eq:three-field}, we set $\eta = 1$, $\alpha = 1000$, $\zeta
= \alpha + 1/3$, $k = 0.9 (1 + \tanh(-2r))$, with $r = \sqrt{x^{2} +
  z^{2}}$, and~$\phi = 0.01$. This test case was studied also
in~\citep{Rhebergen:2014}.
\begin{figure}
  \centering
  \includegraphics[width=0.5\textwidth]{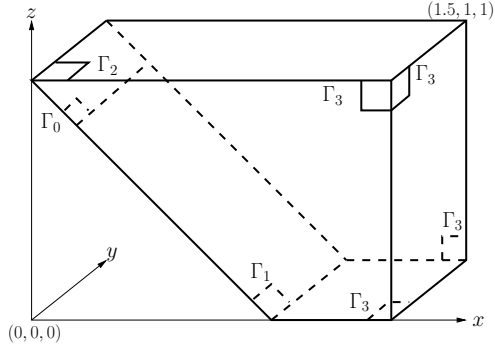}
  \caption{Description of a wedge in a three-dimensional subduction
    zone used for the test cases described in Sections~\ref{ss:tc3}
    and~\ref{ss:tc4}.}
  \label{fig:3Dsubduction}
\end{figure}

We apply the boundary conditions given by~\eqref{eq:bcs} with ${\bf g}
= (1/\sqrt{2}, 0.1, -1\sqrt{2})$ and ${\bf g} = (0, 0, 0)$ on,
respectively, $\Gamma_{1}$ and~$\Gamma_{2}$, and ${\bf g}_N=(0, 0, 0)$
on~$\Gamma_{3}$.

We compute the solution to this test case on three unstructured meshes
with an increasing number of degrees of freedom. We only consider the
AMG-based preconditioners. In Table~\ref{tab:tc3} we present the
number of iterations needed for convergence.  The reduced iteration
counts for the three-field formulation, relative to the two-field
formulation, is clear. This is especially so for the Bi-CGSTAB case.

\begin{table}
  \caption{Iteration counts for the block preconditioners for the
    three-dimensional wedge domain with constant viscosity.}
  \begin{center}
    {\footnotesize
    \begin{tabular}{c|cccc}
      \hline
      DOFs  & $\mathcal{P}_{d2}^{AMG}$/MINRES & $\mathcal{P}_{d3}^{AMG}$/MINRES
      & $\mathcal{P}_{t}^{AMG}$/Bi-CGSTAB & $\mathcal{P}_{t}^{AMG}$/GMRES(100) \\
      \hline
      419,486   & 694 & 366 & 123 & 193 \\
      1,905,881 & 772 & 280 & 71  & 115 \\
      8,493,971 & 766 & 258 & 74  & 144 \\
      \hline
    \end{tabular}
    }
    \label{tab:tc3}
  \end{center}
\end{table}

\subsection{Porosity dependent test case in three dimensions}
\label{ss:tc4}

In this test case we again consider the subduction zone-like domain in
Figure~\ref{fig:3Dsubduction}. We set $\Gamma_{1} = \{{\bf x} \, | \,
x + z = 1, x > 0.1\}$ and solve~\eqref{eq:three-field} with the
constitutive relations in~\eqref{eq:constitutiveRel}, although we
slightly modify the bulk viscosity. For the bulk viscosity we use
\begin{equation}
  \label{eq:mod_bulk_viscosity}
  \zeta^{-1}_{{\rm mod}}
  =
  \begin{cases}
    \zeta^{-1}   & \mbox{if}\ \zeta^{-1} > \zeta^{-1}_c
    \\
    \zeta^{-1}_c & \mbox{otherwise},
  \end{cases}
\end{equation}
with $\zeta^{-1}_{c}$ being a cut-off inverted bulk viscosity. We
choose $\zeta^{-1}_{c} = 10^{-4}$ which roughly corresponds to $\alpha
= 1000$ in~\eqref{eq:mckenzie_old}. Other constants are chosen as $m =
2$, $\lambda = 27$, $r_{\zeta} = 5/3$, $R = 0.1$ and $\phi_0 =
0.05$. A modified bulk viscosity~\eqref{eq:mod_bulk_viscosity} is
needed to prevent the system of equations in~\eqref{eq:three-field}
becoming under-determined; without the modified bulk viscosity
both~\eqref{eq:three-field_b} and \eqref{eq:three-field_c} reduce to
$\nabla \cdot{\bf u} = 0$ in the limit~$\phi \to 0$.

We prescribe the porosity field
\begin{equation}
  \phi =
  (\phi^*-\phi_*) \exp \left(-\frac{(x - x_c)^2 + (y - y_c)^2}{2\omega^2}\right)
  + \phi_*,
  \quad \omega = (\omega^*-\omega_*)\frac{z-1}{z_c - 1}
  + \omega_*,
\end{equation}
with $\phi^{*} = 0.2$, $\phi_{*} = 0$, $x_{c} = 0.3$, $y_{c} = 0.5$,
$z_{c} = 1 - x_{c}$, $\omega^{*} = 0.07$ and $\omega_{*} = 0.01$.  The
porosity field is visualised in Figure~\ref{fig:3Da}. We apply the
boundary conditions in~\eqref{eq:bcs} with ${\bf g} = (1/\sqrt{2},0.1,
-1\sqrt{2})$ on~$\Gamma_1$ and ${\bf g} = (0, 0, 0)$ on~$\Gamma_2$,
and ${\bf g}_N=(0, 0, 0)$ on~$\Gamma_{0}$ and $\Gamma_{3}$. Once the
solution $({\bf u}, p, p_c)$ to~\eqref{eq:three-field} has been
computed, the magma velocity may be recovered from
\begin{equation}
  {\bf u}_f
  = {\bf u}_s - \frac{k}{\phi}\del{\nabla p - {\bf e}_3}.
\end{equation}

We depict the computed magma velocity in Figure~\ref{fig:3Db}.

\begin{figure}
  \centering
  \begin{subfigure}[b]{0.45\linewidth}
    \includegraphics[width=\textwidth]{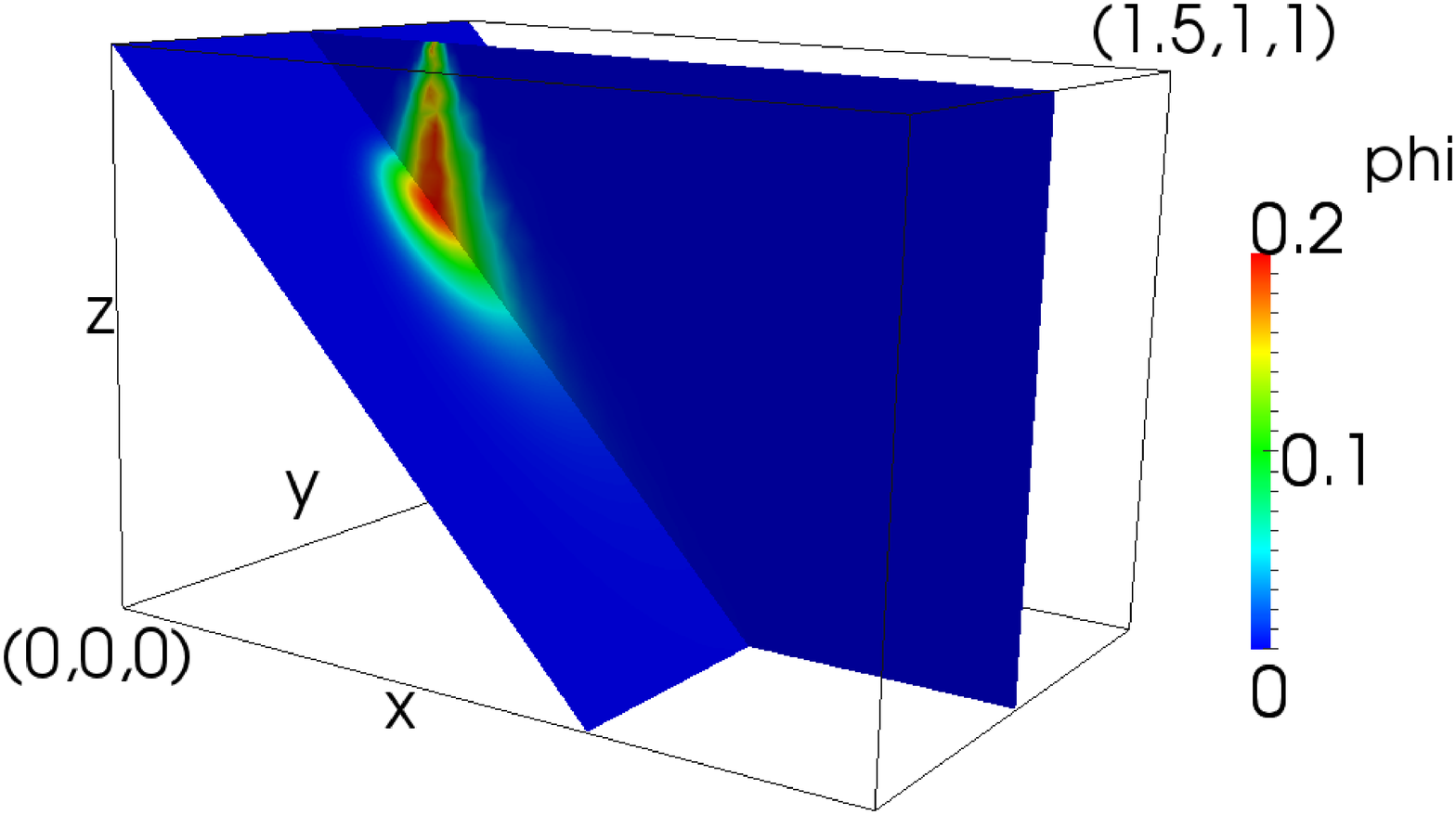}
    \caption{Porosity field.}
    \label{fig:3Da}
  \end{subfigure}
  \begin{subfigure}[b]{0.45\linewidth}
    \includegraphics[width=\textwidth]{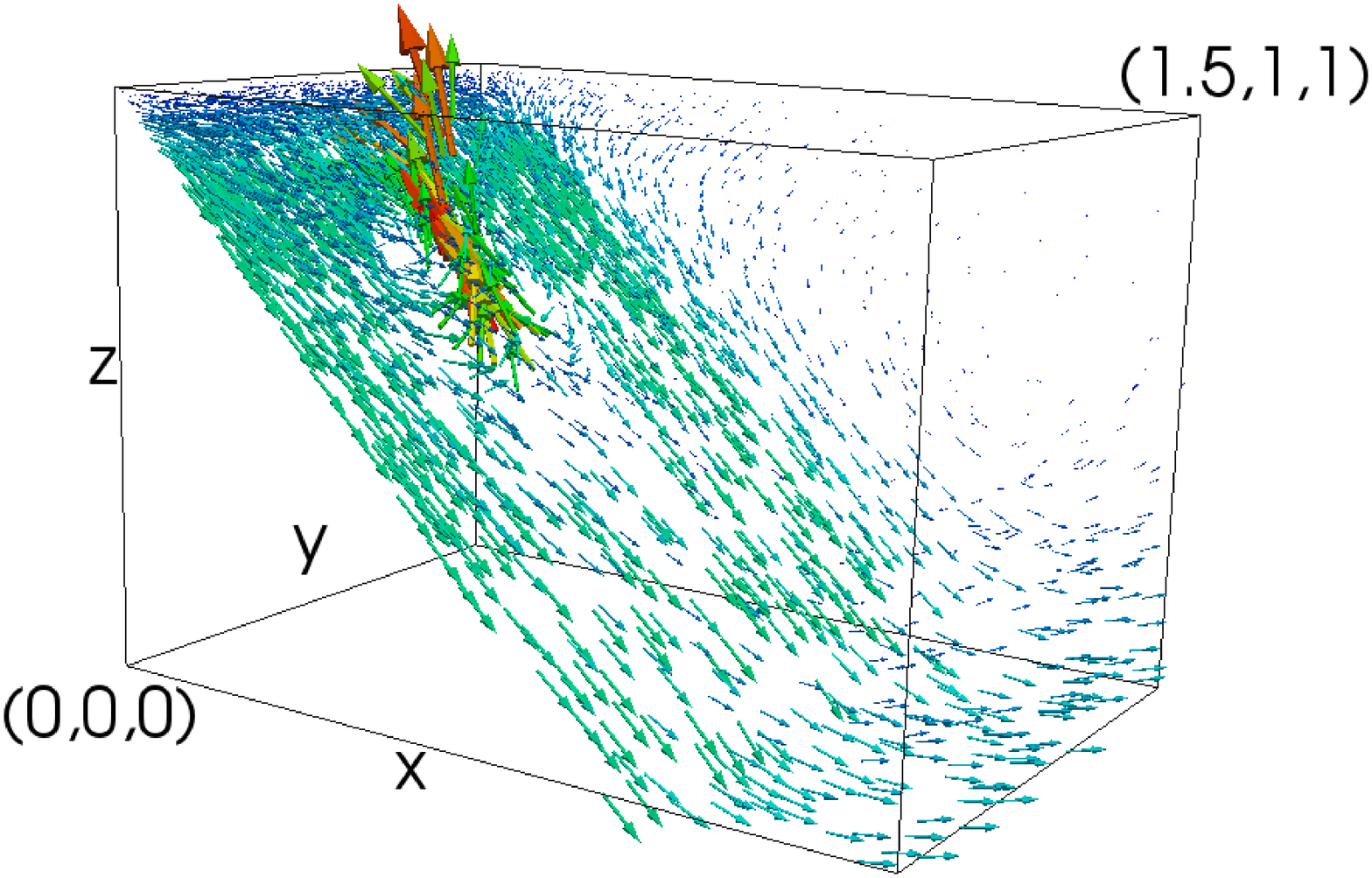}
    \caption{Fluid velocity field.}
    \label{fig:3Db}
  \end{subfigure}
  \caption{The given porosity field and computed fluid velocity in a
    three-dimensional subduction zone. Test case of
    Section~\ref{ss:tc4}.}
  \label{fig:3D}
\end{figure}

We compute the solution to this test case on three unstructured meshes
with differing numbers of degrees of freedom. It was not possible to
compute a converged solution for this test case with the two-field
preconditioner (due to the low porosity) and so we consider the
three-field preconditioners only. We consider only the AMG-based
preconditioners.

Table~\ref{tab:tc4} shows the number of iterations required to meet
the convergence tolerance. We observe no pathological growth in the
iteration count with mesh refinement; on the contrary, the iteration
count tends to drop with mesh refinement. We use unstructured meshes
for this test case, and speculate that the reduced iteration count for
finer meshes could be due to better mesh quality. Noteworthy, again,
is the good performance of Bi-CGSTAB.
\begin{table}
  \caption{Iteration counts for the variable viscosity wedge problem.}
  \begin{center}
    \begin{tabular}{c|ccc}
      \hline
      DOFs  & MINRES & Bi-CGSTAB & GMRES(300) \\
      \hline
      419,486   & 1607 & 320 & 918 \\
      1,905,881 & 1534 & 256 & 609 \\
      8,493,971 &  933 & 225 & 769 \\
      \hline
    \end{tabular}
    \label{tab:tc4}
  \end{center}
\end{table}

\subsection{Summary of numerical simulations}
\label{ss:sumnumsim}

By numerical simulations we have shown that for high
bulk-to-shear-viscosity ratios, the three-field preconditioner is more
robust and has superior performance than the two-field
preconditioner developed in \cite{Rhebergen:2014}. This parameter
regime, which corresponds to low values of porosity, is common in
coupled magma/mantle dynamics simulations. It is exactly in this regime,
where compaction stresses dominate over shear stresses, that the
two-field preconditioner breaks down. The main reason to use a
two-field preconditioner would be because the global system is smaller
than when using the three-field preconditioner. However, for practical
applications, the advantages of introducing the compaction pressure in
the three-field formulation as a new unknown certainly outweigh the
increase in size of the global system.

\section{Conclusions}
\label{s:conclusions}

We have proposed, analysed, and numerically tested new preconditioners
for a three-field formulation of the flow equations for coupled
magma/mantle dynamics. The system of equations can be formulated as a
two-field problem, but it was shown numerically in our past work that
a diagonal block preconditioner using algebraic multigrid was not
uniform with respect to a parameter that modulates compaction
stresses. This motivated the development of preconditioners for a
three-field version of the equations, in which an extra pressure
variable is introduced. Our analysis shows that for a lower block
triangular preconditioner, the eigenvalues of the preconditioned
operator are independent of the problem size, and have a mild
sensitivity to the model parameters. The latter issue is associated
with a degeneracy of the model as porosity approaches zero.  Numerical
experiments indicate that the iteration count for solution of the
three-field problem with the preconditioner developed here does not
grow with problem size and that the sensitivity to model parameters is
small. We therefore expect the preconditioners we have presented to be
effective for large-scale simulation of realistic subduction zones
with large variations in parameters.

\paragraph{Acknowledgements} S.R.~thanks T.~Keller, L.~Alisic and
J.~Rudge for helpful discussions. R.F.K.~thanks the Leverhulme Trust
for support. This work was supported by the Natural Environment
Research Council under grants NE/I026995/1 and NE/I023929/1.

\bibliographystyle{unsrtnat}
\bibliography{references}
\appendix
\section{Non-dimensionalization of the two-phase flow equations}
\label{ap:nondim-2-phase}

The two-phase flow equations that describe coupled magma/mantle
dynamics, as derived by \citet{McKenzie:1984},
are based on mass and momentum conservation. Mass conservation
for the solid (matrix) phase and fluid (melt) phase read
\begin{subequations}
  \label{eq:apmckenzie_mass}
  \begin{alignat}{1}
    \label{eq:apmckenzie_a}
    \partial_{t} \phi - \nabla \cdot \del{(1 - \phi)\mathbf{u}_s} &= 0,
    \\
    \label{eq:apmckenzie_b}
    \nabla\cdot\left({\bf u}_s + {\bf q}\right) &= 0,
  \end{alignat}
\end{subequations}
where $\phi$ is the porosity, $\mathbf{u}_{s}$ is the velocity of the
solid phase and ${\bf q} = \phi\left({\bf u}_f - {\bf u}_s\right)$,
where $\mathbf{u}_{f}$ is the velocity of the melt phase. We have
assumed that there is no melting/freezing and we have taken the
density of the solid and melt phases to be constant and
uniform. Momentum conservation of the melt phase (Darcy's law) reads
\begin{equation}
    \label{eq:apmckenzie_c}
    {\bf q} = -\frac{k}{\mu} \nabla \del{p_f + \rho_{f} gz},
\end{equation}
where $k$ is the permeability, $\mu$ is the fluid viscosity, $p_{f}$
is the pressure in the fluid phase, $\rho_{f}$ is the mass density of
the fluid and $g$ is the constant acceleration due to
gravity. Momentum balance for the two-phase mixture reads
\begin{equation}
  \label{eq:apmckenzie_d}
  -\nabla\cdot\left(2\eta{\bf D}{\bf u}_s\right)
  + \nabla p_f
  =
  \nabla\left(\left(\zeta -
        \frac{2}{3}\eta\right)\nabla\cdot{\bf u}_s\right)
    -\bar{\rho}g{\bf e}_{3},
\end{equation}
where $\eta$ is the shear viscosity of the solid, ${\bf D}{\bf u}_s =
\tfrac{1}{2}\left(\nabla{\bf u}_s + (\nabla{\bf u}_s)^T\right)$ is the
strain rate, $\zeta$ is the bulk viscosity, ${\bf e}_3$ is the
unit vector in the $z$-direction, $\bar{\rho} = \rho_{f} \phi +
\rho_{s}(1 - \phi)$ is the bulk density and $\rho_s$ is the matrix
mass density.

In this paper we are concerned with the efficient solution
of~\eqref{eq:apmckenzie_b}, \eqref{eq:apmckenzie_c}
and~\eqref{eq:apmckenzie_d} for a given porosity field, hence we can
discard~\eqref{eq:apmckenzie_a}. Decomposing the melt pressure as $p_f
= p - \rho_{s}gz$, where $p$ is the dynamic pressure and $\rho_{s}gz$
the `lithostatic' pressure, and substituting~\eqref{eq:apmckenzie_c}
into~\eqref{eq:apmckenzie_b} to eliminate the Darcy flux~${\bf q}$,
we obtain
\begin{subequations}
  \label{eq:apmckenzie_sub}
  \begin{alignat}{1}
    \label{eq:apmckenzie_sub_a}
    -\nabla\cdot\left(2\eta{\bf D}{\bf u}_s\right)
    + \nabla p &=
    \nabla\left(\left(\zeta -
        \frac{2}{3}\eta\right)\nabla\cdot{\bf u}_s\right)
    +g\Delta\rho\phi{\bf e}_3,
    \\
    \label{eq:apmckenzie_sub_b}
    \nabla\cdot{\bf u}_s &=
    \nabla\cdot\left(\frac{k}{\mu}\nabla\left(p-\Delta\rho gz\right)\right),
  \end{alignat}
\end{subequations}
where $\Delta\rho = \rho_s - \rho_f$. Constitutive relations are
required for the permeability and the shear and bulk viscosities. For
now we define
\begin{equation}
  k = k_0k', \quad
  \eta = \eta_0 \eta',\quad
  \zeta = \zeta_0 \zeta',
\end{equation}
where $k_0$, $\eta_0$ and $\zeta_0$ are the characteristic
permeability, shear viscosity and bulk viscosity, respectively, and
$k'$, $\eta'$ and $\zeta'$ are non-dimensional functions that depend
on the porosity~$\phi$.

We non-dimensionalize \eqref{eq:apmckenzie_sub} using
\begin{equation}
  \label{eq:apnondim}
  {\bf u}_s = u_0{\bf u}_s', \
  {\bf x} = H{\bf x}',\
  (\eta, \zeta) = \eta_0(\eta', r_{\zeta}\zeta'),\
  k = k_0k',\
  p=\Delta\rho g H p',
\end{equation}
where primed variables are non-dimensional, $r_{\zeta} =
\zeta_0/\eta_0$, $u_0$ is the velocity scaling given by $u_0 =
\Delta\rho gH^2/\eta_0$ and $H$ is a length scale. Dropping the prime
notation, the two-phase flow equations~\eqref{eq:apmckenzie_sub} in
non-dimensional form are given by
\begin{subequations}
  \label{eq:apmckenzie_nd}
  \begin{alignat}{1}
    \label{eq:apmckenzie_nd_a}
    -\nabla\cdot\left(2\eta{\bf D}{\bf u}_s\right)
    + \nabla p &=
    \nabla\left(\left(r_{\zeta}\zeta -
        \frac{2}{3}\eta\right)\nabla\cdot{\bf u}_s\right)
    +\phi{\bf e}_3,
    \\
    \label{eq:apmckenzie_nd_b}
    \nabla\cdot{\bf u}_s &=
    \nabla\cdot\left(\frac{R^2}{r_{\zeta}+4/3}k\left(\nabla p
        - {\bf e}_3\right)\right),
  \end{alignat}
\end{subequations}
where $R = \delta/H$ and $\delta =
\sqrt{(r_{\zeta}+4/3)\eta_0k_0/\mu}$ is the reference compaction
length~\citep{Takei:2013}.

To simplify the notation, we re-define the non-dimensional
permeability and shear and bulk viscosities as
\begin{equation}
  k \leftarrow \frac{R^2}{r_{\zeta}+4/3}k, \quad
  \eta \leftarrow 2\eta, \quad
  \zeta \leftarrow r_{\zeta}\zeta,
\end{equation}
so that \eqref{eq:apmckenzie_nd} becomes
\begin{subequations}
  \label{eq:apmckenzie_ndn}
  \begin{alignat}{1}
    \label{eq:apmckenzie_ndn_a}
    -\nabla\cdot\left(\eta{\bf D}{\bf u}_s\right)
    + \nabla p
    &=
    \nabla\left(\left(\zeta -
        \frac{1}{3}\eta\right)\nabla\cdot{\bf u}_s\right)
    +\phi{\bf e}_3,
    \\
    \label{eq:apmckenzie_ndn_b}
    \nabla\cdot{\bf u}_s
    &=
    \nabla\cdot\left(k\left(\nabla p
    - {\bf e}_3\right)\right).
  \end{alignat}
\end{subequations}
%
\section{Rates of convergence of the finite element discretization}
\label{ap:femRates}

In this appendix we consider the rate at which the numerical error in
the velocity and fluid pressure fields decreases as function of the
cell size. We use a quadratic polynomial approximation for the
velocity and a linear polynomial approximation for the pressures, and
let~$h$ be a measure of the cell size.

Table~\ref{tab:FEM_p2p1} presents the error in the $L^{2}$ norm for
the two velocity components, $(u,v)$, and fluid pressure, $p$, and
rates at which the errors reduce with decreasing~$h$. We show the
errors and rates of convergence when the solver is terminated at: (a)
a relative preconditioned residual of~$10^{-10}$; and (b) a relative
preconditioned residual of~$10^{-5}$.  When the relative
preconditioned residual reaches $10^{-10}$ we observe that the $L^2$
errors of the velocity and pressure converge at $\mathcal{O}(h^3)$ and
$\mathcal{O}(h^2)$, respectively. However, when we compute the error
with a relative preconditioned residual of only $10^{-5}$, the error
stagnates, with no reduction in the error with mesh refinement.
Similar behaviour is observed for the two field formulation.

\begin{table}
  \begin{center}
    \begin{tabular}{ccccccc}
      \hline
      $N$ & $\norm{u-u_h}_2$ & rate & $\norm{v-v_h}_2$ & rate & $\norm{p-p_h}_2$ & rate \\
      \hline
      &&&&&&\\
      & \multicolumn{6}{c}{Relative preconditioned residual: $10^{-10}$}\\
      &&&&&&\\
      $16^2$& 3.48E-02 & 4.0 & 2.00E-02 & 4.2 & 4.80E-02 & 1.8 \\
      $32^2$& 3.70E-03 & 3.2 & 1.95E-03 & 3.4 & 1.25E-02 & 1.9 \\
      $64^2$& 4.56E-04 & 3.0 & 2.36E-04 & 3.0 & 3.16E-03 & 2.0 \\
      $128^2$& 6.00E-05 & 2.9 & 3.16E-05 & 2.9 & 7.92E-04 & 2.0 \\
      $256^2$& 8.96E-06 & 2.7 & 4.95E-06 & 2.7 & 1.98E-04 & 2.0 \\
      \hline
      &&&&&&\\
      & \multicolumn{6}{c}{Relative preconditioned residual: $10^{-5}$}\\
      &&&&&&\\
      $16^2$& 3.53E-02 & 4.0 & 2.04E-02 & 4.2  & 4.91E-02 & 1.8 \\
      $32^2$& 5.75E-03 & 2.6 & 4.38E-03 & 2.2  & 1.73E-02 & 1.5 \\
      $64^2$& 4.10E-03 & 0.5 & 4.00E-03 & 0.1  & 1.25E-02 & 0.5 \\
      $128^2$& 4.01E-03 & 0.0 & 4.03E-03 & -0.0 & 1.22E-02 & 0.0 \\
      $256^2$& 4.01E-03 & 0.0 & 4.05E-03 & -0.0 & 1.23E-02 & -0.0 \\
      \hline
    \end{tabular}
    \caption{Converge of the finite element solution for
      $P^2$--$P^1$--$P^{1}$ elements on a unit square test with
      $\alpha = 1$ and $(k_{*}, k^{*})=(0.5, 1.5)$ with different
      stopping criteria for the linear solver. Here $(u,v)$ are the
      two velocity components and $p$ is the fluid pressure.}
    \label{tab:FEM_p2p1}
  \end{center}
\end{table}
\section{Two-field preconditioner including viscosity}
\label{ap:twofieldviscosityPreconditioner}

In this appendix we extend the two-field preconditioner
of~\citep{Rhebergen:2014} to include a porosity dependence. The idea
is based on that of \citet{Grinevich:2009}, in which we scale the
pressure mass matrix by the viscosity. The two-field preconditioner is
given by:
\begin{equation}
  \label{eq:twofieldviscosity}
  \mathcal{P}_{2}
  =
  \begin{bmatrix}
    \tilde{K}_{\eta} & 0
    \\
    0 & Q_{\eta} + C_k
  \end{bmatrix},
\end{equation}
where $Q_{\eta}$, $C_k$ and $\tilde{K}_{\eta}$ are the matrices
obtained from, respectively, the discretization of the bilinear forms
$d_{\eta}(\cdot,\cdot)$ in~\eqref{eq:detazeta}, $c(\cdot,\cdot)$
in~\eqref{eq:bilinForms_c} and $\tilde{a}(\cdot, \cdot)$ which is
defined as
\begin{equation}
  \label{eq:Ktilde}
  \tilde{a}({\bf u}, {\bf v})
 =
  \int_{\Omega}\eta{\bf D}{\bf u} : {\bf D}{\bf v} \dif x
  +\int_{\Omega}\left(\zeta -
    \frac{1}{3}\eta\right)(\nabla\cdot{\bf u})(\nabla\cdot{\bf v}) \dif x.
\end{equation}
\end{document}